\pdfoutput=1

\documentclass{amsart}
\usepackage[utf8]{inputenc}
\usepackage[T1]{fontenc}

\usepackage{amssymb}
\usepackage{braket}
\usepackage{mathrsfs}
\usepackage{ifthen}
\usepackage{here}
\usepackage{todonotes}
\usepackage{tikz}
\usetikzlibrary{patterns,decorations.pathreplacing,calligraphy}
\usepackage{comment} 
\usepackage{mleftright}
\usepackage[pagebackref,hypertexnames=false]{hyperref} 
\usepackage{cleveref}
 \usepackage[all]{xy}
 \usepackage{amscd}
 \usepackage[alphabetic,backrefs,msc-links]{amsrefs}
 \usepackage{color}
 \usepackage{enumitem}
\newlist{steps}{enumerate}{1}
\setlist[steps, 1]{label = Step \arabic*:}
\usepackage[abs]{overpic}
\usepackage{tikz-cd}
\usepackage{pinlabel}
\usepackage{amsmath, amsthm,verbatim,amsfonts, graphicx}

\usepackage{geometry}
\makeatletter
\DeclareRobustCommand\widecheck[1]{{\mathpalette\@widecheck{#1}}}
\def\@widecheck#1#2{%
   \setbox\z@\hbox{\m@th$#1#2$}%
   \setbox\tw@\hbox{\m@th$#1%
      {%
         \vrule\@width\z@\@height\ht\z@
         \vrule\@height\z@\@width\wd\z@}$}%
   \dp\tw@-\ht\z@
   \@tempdima\ht\z@ \advance\@tempdima2\ht\tw@ \divide\@tempdima\thr@@
   \setbox\tw@\hbox{%
      \raise\@tempdima\hbox{\scalebox{1}[-1]{\lower\@tempdima\box\tw@}}}%
   {\ooalign{\box\tw@ \cr \box\z@}}}
\makeatother

\theoremstyle{plain}
\newtheorem*{theorem*}{Theorem}
\newtheorem{thm}{Theorem}[section]
\crefname{thm}{Theorem}{Theorems}
\Crefname{thm}{Theorem}{Theorems}
\newtheorem{prop}[thm]{Proposition}
\crefname{prop}{Proposition}{Propositions}
\Crefname{prop}{Proposition}{Propositions}
\newtheorem{lem}[thm]{Lemma}
\crefname{lem}{Lemma}{Lemmas}
\Crefname{lem}{Lemma}{Lemmas}
\newtheorem{cor}[thm]{Corollary}
\crefname{cor}{Corollary}{Corollaries}
\Crefname{cor}{Corollary}{Corollaries}
\newtheorem{rem}[thm]{Remark}
\crefname{rem}{Remark}{Remarks}
\Crefname{rem}{Remark}{Remarks}

\crefname{claim}{Claim}{Claims}
\Crefname{claim}{Claim}{Claims}

\crefname{property}{Property}{Properties}
\Crefname{property}{Property}{Properties}

\crefname{problem}{Problem}{Problems}
\Crefname{problem}{Problem}{Problems}

\crefname{conjecture}{Conjecture}{Conjecture}
\Crefname{conjecture}{Conjecture}{Conjecture}

\theoremstyle{definition}
\newtheorem{defn}[thm]{Definition}
\crefname{defn}{Definition}{Definitions}
\Crefname{defn}{Definition}{Definitions}

\crefname{notation}{Notation}{Notations}
\Crefname{notation}{Notation}{Notations}

\crefname{convention}{Convention}{Conventions}
\Crefname{convention}{Convention}{Conventions}
\crefname{cond}{Condition}{Conditions}
\Crefname{cond}{Condition}{Conditions}

\crefname{assum}{Assumption}{Assumptions}
\Crefname{assum}{Assumption}{Assumptions}

\Crefname{ques}{Question}{Question}

\theoremstyle{remark}
\crefname{rem}{Remark}{Remarks}
\Crefname{rem}{Remark}{Remarks}

\crefname{ex}{Example}{Examples}
\Crefname{ex}{Example}{Examples}

\crefname{section}{Section}{Sections}
\Crefname{section}{Section}{Sections}
\crefname{subsection}{Subsection}{Subsections}
\Crefname{subsection}{Subsection}{Subsections}
\crefname{figure}{Figure}{Figures}
\Crefname{figure}{Figure}{Figures}

\newcommand{\Z}{\mathbb{Z}}

\newcommand{\fraks}{\mathfrak{s}}

\newcommand{\C}{\mathbb{C}}

\newcommand{\R}{\mathbb R}

\newcommand*{\QEDB}{\null\nobreak\hfill\ensuremath{\square}}%

\tikzset{
contains/.style = {draw=none,"\in" description,sloped}
}

\def\det{\mathrm{det}}

\def\dim{\mathrm{dim}}

\def\id{\mathrm{Id}}

\newcommand{\mbar}[1]{{\ooalign{\hfil#1\hfil\crcr\raise.167ex\hbox{--}}}}

\def\wt{\widetilde}
\def\H{\mathbb{H}}

     \RequirePackage{rotating}                   
    \def\HMt{%
       \setbox0=\hbox{$\widehat{\mathit{HM}}$}
       \setbox1=\hbox{$\mathit{HM}$}
       \dimen0=1.1\ht0
       \advance\dimen0 by 1.17\ht1
       \smash{\mskip2mu\raise\dimen0\rlap{%
          \begin{turn}{180}
              {$\widehat{\phantom{\mathit{HM}}}$}
           \end{turn}} \mskip-2mu    
                \mathit{HM}
                    }{\vphantom{\widehat{\mathit{HM}}}}{}}

\title{A satellite formula for real Seiberg--Witten Floer homotopy types}

\author{Jin Miyazawa}
\address{Reserch Institute for Mathematical Sciences, Kyoto University, Japan}
\email{miyazawa.jin.5a@kyoto-u.ac.jp}

\author{JungHwan Park}
\address{Department of Mathematical Sciences, KAIST, Republic of Korea}
\email{jungpark0817@kaist.ac.kr}

\author{Masaki Taniguchi} 
\address{Department of Mathematics, Kyoto University, Japan}
\email{taniguchi.masaki.7m@kyoto-u.ac.jp}

\begin{document}


\begin{abstract}
We establish a satellite formula for the real Seiberg--Witten Floer homotopy types of knots with odd patterns. Using this, we derive several applications to knot concordance theory. The satellite formula follows from a version of the excision theorem for real Floer homotopy types. Additionally, we show that the concordance invariants arising from real Seiberg-–Witten theory depend only on the knot's zero-framed surgery.
\end{abstract}

\maketitle

\section{Introduction} 
Let $\mathcal{C}$ denote the smooth knot concordance group. A knot is called \emph{slice} if it represents the same class as the unknot $U$ in $\mathcal{C}$. For a knot $K$ in the three-sphere $S^3$ and a knot $P$ in the solid torus $S^1 \times D^2$, called a \emph{pattern}, we can construct the \emph{untwisted satellite knot} of $K$, denoted by $P(K)$. This construction induces a map, known as the \emph{satellite operation}, given by $$ P \colon \mathcal{C} \to \mathcal{C} ; \qquad [K] \mapsto [P(K)].$$ In this paper, we investigate satellite operations from the perspective of real Seiberg--Witten theory. 

The \emph{winding number} of a pattern $P$ is defined as the algebraic intersection number of $P$ with a meridional disk of the solid torus. Roughly speaking, we will show that any nontrivial topological phenomena detected by real Seiberg-Witten theory persist under the satellite operation, provided the pattern has an odd winding number. As a result, we obtain the following theorem on satellites of $E_{2,1}$, where $E_{2,1}$ denotes the $(2,1)$-cable of the figure-eight knot throughout the paper.

\begin{thm}\label{thm:topologicalmain} If $P$ is a pattern with an odd winding number and $P(U)$ is slice, then any finite self-connected sum of $P(E_{2,1})$ does not bound a normally immersed disk in the four-ball $B^4$ with only negative double points. In particular, the knot $P(E_{2,1})$ has infinite order in $\mathcal{C}$. \end{thm}

\noindent Here, we say that a surface is \emph{normally immersed} if it is smoothly immersed, with the only singularities being transverse double points in the interior. Note that \Cref{thm:topologicalmain}, in particular, applies to patterns such as $(\text{odd}, 1)$-cables, the Mazur pattern, and their iterates.

For the proof of Theorem~\ref{thm:topologicalmain}, we divide the argument into two main steps. First, we examine the knot $E_{2,1}$, which has garnered significant attention due to its potential as a counterexample to the slice-ribbon conjecture~\cite{Fox:1962, Miyazaki:1994}. Recent work~\cite{DKMPS:2024} established that $E_{2,1}$ is not slice and has infinite order in $\mathcal{C}$, with related results discussed in subsequent works~\cite{ACMPS:2023,KMT:2023, KPT:2024}. We provide a new proof of this and further strengthen the result using real Seiberg-Witten theory~\cite{KMT:2021, KMT:2024}, showing that any finite self-connected sum of $E_{2,1}$ does not bound a normally immersed disk in $B^4$ with only negative double points. We refer the readers to \cite{TW09, Na13, Nak15, Ka22, KMT:2021, Ji22, KMT:2023, Mi23, Li23, BH24, Ba25} for background on real Seiberg–Witten theory. In particular, we use the $K$-theoretic real Fr{\o}yshov invariant $\kappa_R$, introduced by Konno, Miyazawa, and Taniguchi in~\cite{KMT:2021}. In \Cref{Proof of the applications}, we improve the real $10/8$-type inequality established in~\cite[Theorem 1.3]{KMT:2021}, and prove that
\[
\frac{1}{2}  \leq \kappa_R(-E_{2,1}),
\]
where $-E_{2,1}$ denotes the reverse of the mirror image of $E_{2,1}$. We also show that the same inequality holds for any finite self-connected sum of $-E_{2,1}$.

Secondly, we uncover a surprising feature of real Seiberg–Witten theory: for any knot $K$, if $P$ is a pattern with odd winding number and $P(U)$ is slice, then $K$ and $P(K)$ belong to the same local equivalence class of real Floer homotopy type, a notion defined in~\cite{KMT:2024}. This is particularly notable, as an analogous phenomenon fails in many other theories, including Heegaard Floer theory~\cite{OzSz:2004}, involutive Heegaard Floer theory~\cite{Hendricks-Manolescu:2017}, Khovanov homology theory~\cite{Khovanov:2000}, equivariant singular instanton Floer theory~\cite{DS19, DISST2022}, and equivariant Seiberg–Witten theory~\cite{BH21, BH22, IT2024}.\footnote{These theories provide slice-torus invariants~\cite{Livingston:2004, Oz-Sz:2003, Rasmussen:2010, DISST2022, IT2024}. Since an $(n,1)$-cable of a positive torus knot is strongly quasipositive for any $n > 1$~\cite{Rudolph:1993, Rudolph:1998}, it follows that the slice-torus invariants of $K$ and $P(K)$ differ when $K$ is a positive torus knot and $P$ is an $(n,1)$-cable with $n > 1$. In fact, the ability of these theories to distinguish between $K$ and $P(K)$ has led to the discovery of many interesting phenomena~\cite{Levine:2016, KP:2018, FPR:2019, HKPM:2022, Collins:2022}.}  Moreover, we remark that the assumption of an odd winding number is necessary.  
For an even winding number, the relation remains unclear.  
For instance, while $U$ and its $(2,1)$-cable $U_{2,1}$ both have trivial local equivalence classes, $E$ and $E_{2,1}$ have distinct local equivalence classes~\cite{KMT:2021, KMT:2024, KPT:2024}.



The fact that $K$ and $P(K)$ belong to the same local equivalence class can be expressed as follows: if $P$ is a pattern with an odd winding number and $P(U)$ is slice, then
$$ P_* \colon \mathcal{LE} \to \mathcal{LE} ; \qquad [ SWF_R (K)]_{\text{loc}} \mapsto [SWF_R (P(K))]_{\text{loc}}$$
is the identity map, where $\mathcal{LE}$ denotes the local equivalence group of real Seiberg-Witten theory and $[SWF_R(K)]_{\text{loc}}$ denotes the local equivalence class of the real Floer homotopy type of $K$. For any knot $K$ and a nonnegative integer $n$, let $nK$ represent the connected sum of $K$ with itself $n$ times. Then the fact that $P_*$ acts as the identity map implies that
$$n[SWF_R(P(mK))]_{\text{loc}} = n[SWF_R(mK)]_{\text{loc}}=nm[SWF_R(K)]_{\text{loc}} \in \mathcal{LE},$$ for all nonnegative integers $n$ and $m$. In particular, any conclusion drawn about a knot $mK$ and its finite self-connected sum also applies to $P(mK)$ and its finite self-connected sum. Specifically, the conclusion of Theorem~\ref{thm:topologicalmain} extends to $P(mE_{2,1})$ for any positive integer $m$. Moreover, the assumption that $P(U)$ is slice in Theorem~\ref{thm:topologicalmain} is, in fact, unnecessary. Combining the above discussions, the largest family of knots obtained from satellites of $E_{2,1}$, for which we can conclude that they have infinite order in $\mathcal{C}$, is described as follows:

\begin{thm}\label{thm:topologicalmaingeneral}
If $P$ is a pattern with an odd winding number, then $P(mE_{2,1}) \mathbin{\#} -P(U)$ has infinite order in $\mathcal{C}$, where $m$ is any nonzero integer.
\end{thm}
\noindent In slightly more detail, the proof proceeds as follows. Using the equivalence of local equivalence classes of knots and their satellites, together with the inequality for $\kappa_R(-mE_{2,1})$, we conclude that for any winding number odd pattern $\overline{P}$ and any positive integer $m$,
\[
\frac{1}{2} \leq \kappa_R(-mE_{2,1}) = \kappa_R\left(\overline{P}(-mE_{2,1}) \mathbin{\#} -\overline{P}(U)\right),
\]
and the same inequality holds for any finite self-connected sum of such satellites. The inequality itself is of independent interest, as the invariant $\kappa_R$ is known to be notoriously difficult to compute.

We make the following two remarks. First, we consider the example $E_{2,1}$ to emphasize that, given the current state of other theories, the conclusions in this paper cannot be derived from them. In general, analogous conclusions hold for any knot representing a nontrivial class in $\mathcal{LE}$. For instance, consider the positive torus knot $T_{3,11}$. For a knot $K$ with $\mathrm{Arf}(K) = 0$, the \emph{stabilizing number} $\mathrm{sn}(K)$ is defined as the minimum number of copies of $S^2 \times S^2$ that must be connected summed to $B^4$ for $K$ to bound a smooth null-homologous disk in it. It was shown in \cite[Theorem 1.11]{KMT:2021} that $\mathrm{sn}(K)$ can differ arbitrarily from its topological counterpart, the \emph{topological stabilizing number} $\mathrm{sn}^{\mathrm{Top}}(K)$, by taking finite self-connected sums of $T_{3,11}$ and applying real Seiberg–Witten theory along with results from \cite{McCoy:2021, FMP:2022, BBL:2020}. Our work readily implies the following:

\begin{thm} \label{thm:stabnumber} 
There exists a knot $K$ in $S^3$ with $\mathrm{Arf}(K) = 0$, such as $T_{3,11}$ or any of its finite self-connected sums, for which the limit  
\[
\lim_{n \to \infty} \left( \mathrm{sn}(n P(K)) - \mathrm{sn}^{\mathrm{Top}}(n P(K)) \right) = \infty
\]
holds for any pattern $P$ with odd winding number where $P(U)$ is the unknot.
\end{thm}

Secondly, the truly interesting aspect of these results lies in their conclusion for finite self-connected sums of satellites. For example, when the pattern is the Mazur pattern $Q$, it is known that the satellite knot $Q(K)$ bounds a smoothly embedded disk in a homology 4-ball if and only if $K$ does. This result follows from the fact that the zero-framed surgery on the satellite knot $Q(K)$ is smoothly homology cobordant to the zero-framed surgery on $K$ relative to their meridians~\cite[Corollary 2.2]{CFHH:2013} (see also~\cite[Corollary 5.2]{CDR:2014}). However, this equivalence does not extend to nontrivial finite self-connected sums of satellites, as the satellite operation is not a homomorphism in general (see e.g.\ \cite{Miller:2023}), and there is no direct relationship between the zero-framed surgery of a knot and that of its finite self-connected sum.



The following theorem is the key ingredient underlying all the results presented thus far:

\begin{thm}\label{thm:mainhomotopytype} Let $K$ be a knot in $S^3$. If $P$ is a pattern with an odd winding number, then the real Floer homotopy types \[ SWF_R(P(K)) \qquad \text{and} \qquad SWF_R(K) \wedge SWF_R(P(U)) \] are $\mathbb{Z}_4$-equivariantly stably homotopy equivalent. 
\end{thm}
\noindent This follows from a real version of the excision theorem, given in \cref{exthm}, where the original theorem was proved by Floer in \cite{floer1990instanton}. We adopt Kronheimer and Mrowka's approach to the excision theorem from \cite{kronheimer2010knots}. In particular, the homotopy equivalences in \cref{thm:mainhomotopytype} arise as the real Bauer–Furuta invariants of the excision cobordisms used in \cite{kronheimer2010knots}.

In \cite{KMT:2021, KMT:2024}, the authors defined the concordance invariants \( \delta_R, \overline{\delta}_R, \underline{\delta}_R \), and \( \kappa_R \) from the real Floer homotopy type \( SWF_R(K) \). As an immediate consequence of \cref{thm:mainhomotopytype}, we obtain:

\begin{cor}\label{cor:mainhomotopytype}
Let $K$ be a knot in $S^3$. If $P$ is a pattern with an odd winding number, then the knots $K$ and $P(K) \# -P(U)$ share the same values for the concordance invariants:
\[
\delta_R, \quad \overline{\delta}_R, \quad \underline{\delta}_R, \quad \text{and} \quad \kappa_R. \eqno\QEDB
\] 
\end{cor}

Furthermore, we define a real Seiberg–Witten Floer homotopy type for homology $S^1 \times S^2$’s and show that $SWF_R(S^3_0(K))$ and $SWF_R(K)$ are locally equivalent, where $S^3_0(K)$ denotes the zero-framed surgery on a knot $K$.  Recall that many concordance invariants are known \emph{not} to be determined by the zero-framed surgery on a knot~\cite{Yasui:2015} (see also~\cite{CFHH:2013,Piccirillo:2019,HMP:2021}). Examples include the $\tau$, $\nu$, $\nu^+$, and $\varepsilon$ invariants from Heegaard Floer theory~\cite{Oz-Sz:2003, OzSz11, Hom:2014, Hom-Wu:2016-1}, as well as the $s$ invariant from Khovanov homology~\cite{Rasmussen:2010} and, more broadly, all slice-torus invariants~\cite{Lobb09, Wu09, Lo12, Le14, LS14, LL:2016,GLW19, BS21, SS22,   DISST2022, IT2024}. In contrast, the local equivalence between $SWF_R(S^3_0(K))$ and $SWF_R(K)$ implies that all concordance invariants arising from the local equivalence class of $SWF_R(K)$ are determined by the zero-framed surgery on~$K$.

\begin{thm}\label{0surgerydetermine}
    Let $K$ be a knot in $S^3$. 
    Then, the local equivalence class of $SWF_R(K)$ is determined by the orientation preserving diffeomorphism type of $S^3_0(K)$.
    In particular, the concordance invariants 
    \[
    \delta_R(K), \quad \overline{\delta}_R(K), \quad \underline{\delta}_R(K), \quad \text{and} \quad \kappa_R(K)
    \]
    depend only on the orientation preserving diffeomorphism type of $S^3_0(K)$.
\end{thm}



Furthermore, since the homotopy type of $SWF_R(P(K))$ can be computed via \cref{thm:mainhomotopytype}, the associated cohomological invariants are also determined. In particular, the degree-type invariant  
\[
|\deg(K)| \in \mathbb{Z}_{\geq 0}
\]  
defined as the absolute value of the signed count of $\{\pm 1\}$-framed real Seiberg–Witten solutions on the double branched cover of a knot $K$ with respect to its unique spin structure (see \cite[Definition 4.28]{Mi23}; see also \cite{KPT:2024}), can be computed as follows:


\begin{thm}\label{degree}
    For any knot $K$ in $S^3$ and any pattern $P$ with an odd winding number, we have  
    \[
    |\deg (K)|\cdot |\deg P(U)| = |\deg (P(K))| .
    \]
\end{thm}

\noindent Note that the theorem also allows us to compute Miyazawa's $2$-knot invariants~\cite[Definition 3.4]{Mi23} for $2$-knots obtained via twisted roll spinning~\cite[Section 4]{Mi23}.


Finally, we prove that Miyazawa's $2$-knot invariants depend only on the orientation preserving diffeomorphism type of the surgery manifold. 
\begin{thm}
    Let $K$ be a smooth $2$-knot in $S^4$.  
    Then, Miyazawa's degree invariant $|\deg(K)|$ depends only on the orientation preserving diffeomorphism type of the $4$-manifold obtained by performing the surgery along $K$.  
\end{thm}

\noindent For the proof, in analogy with real invariants of smooth homology $S^1 \times S^2$ manifolds, we define an invariant  
$|\deg(X)| \in \mathbb{Z}_{\geq 0}$  
for any oriented homology $S^1 \times S^3$ manifold $X$, and show that  
\[
|\deg(X(K))| = |\deg(K)|,
\]
where $X(K)$ denotes the manifold obtained by performing surgery along a 2-knot $K$  in $S^4$.
Using the invariant $|\deg (X)|$, we also obstruct the existence of a positive scalar curvature metric on certain homology $S^1 \times S^3$ manifolds. See \cref{homologys1s3} for further discussion.

\subsection*{Acknowledgements} 
This joint work began at the conference \emph{The East Asian Conference on Gauge Theory and Related Topics II}, held in Nara. We thank the organizers for inviting all of the authors and for providing a stimulating environment. The first author is partially supported by JSPS KAKENHI Grant Number 24K22832. 
The second author is partially supported by the Samsung Science and Technology Foundation (SSTF-BA2102-02) and the NRF grant RS-2025-00542968.
The third author was partially supported by JSPS KAKENHI Grant Number 22K13921.

\subsection*{Notation and conventions}
Throughout this paper, all $3$-manifolds are assumed to be smooth, connected, closed, and oriented, and all $4$-manifolds are smooth, connected, compact, and oriented. All maps between manifolds are smooth. Given a manifold $Y$, we denote its reverse orientation by $-Y$. Homology is taken with integral coefficients unless specified otherwise. For any $\text{spin}^c$ structure $\mathfrak{s}$, we denote its conjugate $\text{spin}^c$ structure by $\overline{\mathfrak{s}}$.

\section{Real structures on zero-framed surgery trace} 
In this section, we discuss real spin$^c$ and real spin structures on the zero-framed surgery traces of knots.

\subsection{Real spin$^c$ structures} 
First, we review the definitions of real spin$^c$ and real spin structures and classifications of them. As a related reference see, \cite{Na13, Ji22, KMT:2024,  Ba25}.

Let $X$ be an oriented smooth $3$- or $4$-manifold, and let $\fraks$ be a spin$^c$ structure on $X$.  
Let $\tau \colon X \to X$ be a smooth involution that preserves the orientation of $X$ and satisfies  
\[
\tau^{\ast}\fraks \cong \bar{\fraks}.
\]  
Additionally, we fix an $\tau$-invariant metric on $X$. We identify $\fraks$ with its spinor bundle $\mathbb{S}$, equipped with the Clifford multiplication  
\[
\rho \colon \Lambda T^*X \otimes \C \to \operatorname{End}(\mathbb{S}).
\]


\begin{defn}\label{def_real_spinc} Let $n$ be either $3$ or $4$.  
Let $X$ be an $n$-manifold with a spin$^c$ structure $\mathfrak{s}$. We denote by $\mathbb{S}$ the spinor bundle associated with $\mathfrak{s}$ and by $\rho$ its Clifford multiplication. A {\it real structure} on $\mathbb{S}$ an anti-complex linear map
\[
I\colon \mathbb{S} \to \mathbb{S} 
\]
that satisfies the following conditions:
\begin{itemize}
    \item[(i)] $I$ covers $\tau$, 
    \item[(ii)] $I(\rho(\xi)\phi)=\rho(\tau^*(\xi))I(\phi)$ for all $\xi \in T_x X$ and $\phi \in \mathbb{S}_x$ 
    \item[(iii)] $I^2 =\id$, 
    \item[(iv)] $I$ preserves the hermitian metric on $\mathbb{S}$.
\end{itemize}
Two real structures $(\mathbb{S}, \rho, I )$ and $(\mathbb{S}', \rho', I' )$ are isomorphic if there is spin$^c$ isomorphism $\phi \colon \mathbb{S} \to \mathbb{S}'$ which commutes with $I$ and $I'$. 
\end{defn}


 We will see that the notion of a real spin$^c$ structure on $X$ with a free involution $\tau$ corresponds to the notion of a spin$^{c-}$ structure on $X/\tau$, which was introduced by Nakamura in \cite{Na13}.
\begin{defn}
Let $n$ be either $3$ or $4$.  
Let $M$ be an oriented Riemannian $n$-manifold, and let $l$ be a $\mathbb{Z}$-local coefficient on $M$. We call $(P, \pi)$ a spin$^{c-}$ structure on $(M, l)$ if the following conditions are satisfied:
        \begin{itemize}
            \item[(i)] Let $Fr(TM)$ be a oriented orthonormal frame bundle. Then there is a smooth map $\pi \colon P \to Fr(TM)$ such that for all $p \in P$ and $[g, u] \in \rm{Spin}(n) \times_{\Z_2} \rm{Pin}^-(2)$, we have 
            \[
            \pi\left(p\cdot [g, u]\right)=\pi(p)\pi_0(g)
            \]
            where $\pi_0$ is the projection $\rm{Spin}(n) \to \rm{SO}(n)$. 
            \item[(ii)] Let $\rho_0 \colon \rm{Pin}^-(2) \to \rm{O}(2)$ be the projection and $\rho'_0\colon \rm{Spin}^{c-}(n) \to \rm{O}(2)$ be a map $[g, u] \mapsto \rho_0(u)$. Then the associated bundle $E=P \times_{\rho'_0}\R^2$ satisfies 
            \[
                \det(E) \cong l \otimes_{\Z}\R
            \] as real line bundles. 
        \end{itemize} 
        Two spin$^{c-}$ structures $(P, \pi)$ and $(P', \pi')$ are isomorphic if there is an isomorphism $f \colon P \to P'$ of principal $\rm{Spin}^{c-}(n)$-bunldle such that $\pi=\pi'\circ f$.
\end{defn}
   A nearly complete proof of the following lemma appears in \cite[Section 2.5]{Nak15}.

\begin{lem}\label{equivalence_defn}Let $n$ be either $3$ or $4$.  
    Let $X$ be a $n$-manifold with Riemannian metric. 
    Let $\tau \colon X \to X$ be a free involution and let $l :=X \times_{\Z_2}\Z$ be a $\Z$-local system on $M:=X/\tau$. Then there is a one-to-one correspondence between the set of isomorphism class of real spin$^c$ structure on $(X, \tau)$ and the set of isomorphism class of spin$^{c-}$ structure on $(M, l)$. 
\end{lem}
\begin{proof}
    We prove this theorem in the case that $n=4$ since the case $n=3$ is similar.  
    Firstly, let us recall the definition of the real spin$^c$ structure on $X$ using a principal $\rm{Spin}^c (4)$-bundle, while we used a complex vector bundle with Clifford multiplication in \cref{def_real_spinc}. 
    Let $P$ be a principal $\rm{Spin}^c (4)$-bundle on $X$ and let $\pi \colon P \to Fr(X)$ be a projection to the orthonormal oriented frame bundle. The spinor bundle, which is a complex vector bundle with Clifford multiplication, is an associated vector bundle of $P$. 
    Conversely, if a rank $4$ complex vector bundle $\mathbb{S}$ with Clifford multiplication is given, then $P_x$ is given by a subset of the unitary frame $\langle u_1, \dots, u_4 \rangle$ of $\mathbb{S}_x$ such that there exist an oriented orthonormal frame $\langle e^1,\dots, e^4 \rangle$ of $T^*_xX$ and the Clifford multiplication is represented by the Gamma matrices with the basis $\langle u_1, \dots, u_4 \rangle$ (see \cite[p.2]{kronheimer2007monopoles}). 
    Therefore these two definitions are equivalent. 
    Using the principal $\rm{Spin}^c (4)$-bundle $P$ over $M$, let us prove that our definition of real spin$^c$ structure is equivalent to the triple $(P, \pi, J)$ such that $J \colon P \to P$ is a lift of $\tau$ with 
    \[
    J^2=-1 \qquad \text{ and }\qquad J(p)[g, u]=J(p[g, \bar{u}])
    \]
    where $g \in \rm{Spin(4)}$ and $u \in \rm{U}(1)$. If we have such $J$ on $P$, then the involution $I$ on $\mathbb{S}$ is given as follows. The georup $\rm{Spin}^c(4)$ is isomorphic to $(\rm{Sp}(1) \times \rm{Sp}(1) \times U(1))/\{\pm 1\}$ and let $\rho_{\pm} \colon (\rm{Sp}(1) \times \rm{Sp}(1) \times U(1))/\{\pm 1\} \to \rm{U}(2)$ be a representation given by
    \[
        \rho_{\pm}([(g_+, g_-, u)])\phi=g_{\pm}\phi u^{-1}
    \]
    where $\phi \in \H \cong \C \oplus j\C \cong \C^2$, $g_{\pm} \in \rm{Sp}(1)$ and $u \in \rm{U}(1)$. Using these representations, the spinor bundle is given by the associated bundles $P \times_{\rho_{\pm}} \H=\mathbb{S}^{\pm}$ and $\mathbb{S}=\mathbb{S}^+ \oplus \mathbb{S}^-$. 
    Then we define $I$ by $I([p, \phi])=[J(p), \phi j]$ . One can check that this $I$ satisfies the priperties (i)$\sim$(iv) in \cref{{def_real_spinc}}. 
    Let us assume that a triple $(\mathbb{S}, \rho,I)$ in \cref{def_real_spinc} is given. Then $J$ is given by 
    \[
        J\langle u_1, u_2, u_3, u_4 \rangle=
        \langle I(u_2), -I(u_1), I(u_4), -I(u_3) \rangle. 
    \]
    One can check that the Clifford multiplication of $\tau^*e^1, \ldots,\tau^*e^4$ is given by the Gamma matrices with basis \[
    \langle I(u_2), -I(u_2), I(u_4), -I(u_3) \rangle.
    \]
    In \cite[Section 2.5]{Nak15} the following one-to-one correspondence are proved: the set of isomorphism class of spin$^{c-}$ structure on $(M, \tau)$ corresponds to the set of the isomorphism class of $(P, \pi, J)$. This completes the proof. 
\end{proof}

\begin{rem}
    Although the proof of \cref{equivalence_defn} can be extended to any dimension, we focus on the 3- and 4-dimensional cases in this paper.  
\end{rem}

We define real spin structure as follows: 
\begin{defn} 
Let $n$ be either $3$ or $4$.  
Let $X$ be an $n$-manifold with a spin structure $\mathfrak{s}$. We denote by $\mathbb{S}$ the spinor bundle associated with $\mathfrak{s}$ and by $\rho$ its Clifford multiplication.  
A {\it real structure} on $\mathbb{S}$ is an anti-complex linear map  
\[
I\colon \mathbb{S} \to \mathbb{S}
\]
that satisfies the following conditions:
\begin{itemize}
    \item[(i)] $I$ covers $\tau$, 
    \item [(ii)]$I(\rho(\xi)\phi)=\rho(\tau^*(\xi))I(\phi)$ for all $\xi \in T_x X$ and $\phi \in \mathbb{S}_x$ 
    \item[(iii)] $I^2 =\id$, 
    \item[(iv)] $I$ preserves the hermitian metric on $\mathbb{S}$, 
    \item[(v)] $I$ commute with $j \in \H$.
\end{itemize}
Two real structures $(\mathbb{S}, \rho, I )$ and $(\mathbb{S}', \rho', I' )$ are isomorphic if there is spin isomorphism $\phi \colon \mathbb{S} \to \mathbb{S}'$ which commutes with $I$ and $I'$. 
\end{defn}

There is a principal bundle formulation of real spin structure. 

\begin{defn}
Let $n$ be either $3$ or $4$.  
Let $M$ be an oriented Riemannian $n$-manifold, and let $l$ be a $\mathbb{Z}$-local coefficient on $M$. We call $(P, \pi)$ a spin$^{-}$ structure on $(M, l)$ if the following conditions are satisfied:
        \begin{itemize}
            \item[(i)] Let $Fr(TM)$ be a oriented orthonormal frame bundle. Then there is a smooth map $\pi \colon P \to Fr(TM)$ such that for all $p \in P$ and $[g, u] \in \rm{Spin}(n) \times_{\Z_2} \langle j \rangle $, we have 
            \[
            \pi(p\cdot [g, u])=\pi(p)\pi_0(g)
            \]
            where $\pi_0$ is the projection $\rm{Spin}(n) \to \rm{SO}(n)$. 
            \item[(ii)] Let $\rho_0 \colon \Z_4 \to  \Z_2$ be the projection and $\rho'_0\colon \rm{Spin}^{-}(n) \to \rm{O}(1)$ be a map $[g, u] \mapsto \rho_0(u)$. Then the associated bundle $E=P \times_{\rho'_0}\R$ satisfies 
            \[
                E \cong l \otimes_{\Z}\R
            \] as real line bundles. 
        \end{itemize} 
        Two spin$^{-}$ structures $(P, \pi)$ and $(P', \pi')$ are isomorphic if there is an isomorphism $f \colon P \to P'$ of principal $\rm{Spin}^{-}(n)$-bunldle such that $\pi=\pi'\circ f$.
\end{defn}

The following Lemma is proved by the similar way of the proof of  \cref{equivalence_defn}. 
\begin{lem}\label{equivalence_defn_spin}Let $n$ be either $3$ or $4$.  
    Let $X$ be a $n$-manifold with Riemannian metric. 
    Let $\tau \colon X \to X$ be a free involution and let $l :=X \times_{\Z_2}\Z$ be a $\Z$-local system on $M:=X/\tau$. Then there is a one-to-one correspondence between the set of isomorphism class of real spin structure on $(X, \tau)$ and the set of isomorphism class of spin$^{-}$ structure on $(M, l)$. \qed
\end{lem}

\subsubsection{Dimension four}
The following lemma describes the existence and the classification of real spin$^c$ or real spin structures on involutive 4-manifolds. 

\begin{lem}\label{lem:spinc4manifold}
Suppose $X$ is an oriented $4$-manifold possibly with boundary equipped with a smooth involution  
\[
\tau \colon X \to X.
\]
    Regarding the existence, we have  
    \begin{itemize}
    \item If $\tau$ has a non-empty fixed point set of codimension $2$, the condition $$\tau^* \mathfrak{s} \cong \overline{\mathfrak{s}} \qquad  \text{ and } \qquad H^{1}(X, \Z)^{-\tau^*}=0$$ is sufficient to ensure the existence of a real structure on $\mathfrak{s}$.
    \item If $\tau$ is free, then $X$ is a regular double cover of $X/\tau$. In this case, the existence of an $O(2)$-bundle $E$ over $X/\tau$ such that  
    \begin{align*}
   w_2(X/\tau) + w_2(E)+ w_1(E)^2 = 0 \qquad \text{ and }\qquad 
    w_1(E) = w_1(l)
    \end{align*}
    is sufficient to ensure the existence of a real spin$^c$ structure on $X$, where $l$ is the local system with fiber $\mathbb{Z}$ induced by the double cover. 
\end{itemize}
For the classification, if we fix a spin$^c$ structure $\mathfrak{s}$ on $X$, there is a one-to-one correspondence between isomorphism classes of real structures and  
\[
H^2(X)^{-\tau^*} \oplus \frac{H^1(X)^{\tau^*}}{\operatorname{im}(1+\tau^* \colon H^1(X) \to H^1(X))}
\]
in the non-free case. In the free case, the classification is given by  
\[
H^2(X/\tau; l).
\]
\end{lem}

\begin{proof}
For the existence, the first claim is proved in \cite[Lemma 2.9]{KMT:2024}, and the second claim is proved in \cite[Proposition 3.4]{Na13}.  
For the classification, see \cite[Section 3.1]{Ji22} and \cite[Proposition 2.3]{Nak15}.   Here, we use the fact that, in the free case, the existence and classification results for real spin$^c$ structures on $X$ coincide with those for spin$^{c-}$ structures on $X/\tau$, by \cref{equivalence_defn}.
\end{proof}






\begin{lem}\label{classification:real_spin}
Suppose $X$ is an oriented $4$-manifold possibly with boundary equipped with a smooth involution  
\[
\tau \colon X \to X.
\] 
Suppose $\mathfrak{s} $ is a spin structure on $X$.
    Regarding the existence, we have  
    \begin{itemize}
    \item[(i)] If $\tau$ has a non-empty fixed point set of codimension $2$, then  the condition
    \[
    \tau^* \mathfrak{s} \cong \mathfrak{s}
    \]  
    is a sufficient condition for $\mathfrak{s}$ to admit a real spin structure. In this case, $\tau$ is called an \emph{odd involution}.
    \item[(ii)] If $\tau$ is free, then $X$ is a regular double cover of $X/\tau$. In this case, the condition  
    \[
    w_2(X/\tau) + w_1(l)^2 = 0
    \]  
    is a sufficient condition for $X$ to admit a real spin structure. Here, $l$ is the real line bundle  
    \[
    l = X \times_{\mathbb{Z}_2} \mathbb{R} \to X/\tau.
    \]  
    In this case, $X/\tau$ also admits a real spin structure.
\end{itemize}
For the classification, if a real spin structure exists, its classification is given by the isomorphism classes of real line bundles with an order $2$ lift of $\tau$ and a choice of the sign of $I$.  
In particular, if we fix a spin structure $\mathfrak{s}$ on $X$, there are exactly two choices of real spin structures on it.

\end{lem}

\begin{rem}    The two choices of the real spin structure are given as $I$ and $-I$. A choice of such data does not change real Seiberg--Witten theory since multiplication by $i \in U(1)$ gives a bijective correspondence between these configuration spaces. Therefore, the real Floer homotopy types or the real Bauer--Furuta invariants do not depend on these choices essentially. We do not care these choices in this paper. 

\end{rem}

\begin{proof}
   Note that in the case of a free involution, we can identify real spin structures on $X$ with spin$^{c-}$ structures whose structure group reduces to  
\[
\mathrm{Spin}(n) \times_{\pm 1} \{\pm 1, \pm j\} \subset \mathrm{Spin}(n) \times_{\pm 1} \mathrm{Pin}^-(2)
\]
from \cref{equivalence_defn_spin}. 

For the existence of real spin structures, the item (i) is proved in \cite[Lemma 4.2]{Ka22}, and the item (ii) is proved in \cite[Remark 3.5]{Na13}. Next, we classify real spin structures. Recall that the classification of spin structures is given by real line bundles when a spin structure $\mathfrak{s}$ is fixed. The correspondence is given by tensoring a line bundle $l$ with the spinor bundle $\mathbb{S}$ of $\mathfrak{s}$, where the $j$-action is given by $j \otimes \operatorname{id}_l$.   If there exists an order-two lift $\widetilde{\tau}$ of $\tau$ on $l$, then  
\[
I' = \pm I \otimes \widetilde{\tau}
\]  
defines a real spin structure. Now, assume that $\mathfrak{s}' = \mathfrak{s} \otimes l$ admits a real spin structure $I'$. Let $\mathbb{S}'$ be a spinor bundle for $\mathfrak{s}'$. We now consider the line bundle
\[
\left\{ f \colon \mathbb{S} \to \mathbb{S}' \,\middle|\, 
\begin{array}{l}
f \text{ is a } \mathbb{C}\text{-linear bundle map}, \\
f \circ \rho(\xi) = \rho'(\xi) \circ f \quad \text{for all } \xi \in T^*X, \\
f \circ j = j \circ f
\end{array}
\right\},
\]
which is actually isomorphic to $l$. Also, this line bundle $l$ admits a lift of $\tau$ defined by  
\[
f \mapsto I \circ f \circ I',
\]
which has order two. This completes the proof. 
\end{proof}

\subsubsection{Dimension three}
Next, we discuss real spin$^c$ and spin structures on 3-manifolds.  
Note that if we have a real spin or spin$^c$ structure $(\mathbb{S}, I)$ on a 3-manifold $Y$ with an involution $\tau$, we can construct a real spin or spin$^c$ structure on $[0,1] \times Y$ by setting  
\[
\mathbb{S}^+ = \pi^* \mathbb{S} \qquad \text{and} \qquad \mathbb{S}^- = \pi^* \mathbb{S},
\]
with the natural Clifford multiplication and an involution given by the pullback of $I$.
 For the formula of Clifford multiplication, see \cite[page 89]{kronheimer2007monopoles}.   On the other hand, if we are given a real spin or spin$^c$ structure $(\mathbb{S}, I)$ on $[0,1] \times Y$, the Clifford multiplication provides an identification  
\[
\rho(dt) \colon \mathbb{S}^+ \to \mathbb{S}^-.
\]
Moreover, $\mathbb{S} := \mathbb{S}^+|_{\{1\} \times Y}$ is equipped with the induced Clifford multiplication.   If we assume that the involution $\tau \colon [0,1] \times Y \to [0,1] \times Y$ preserves the time direction, then we obtain a naturally induced real spin$^c$ or spin structure  
\[
I \colon \mathbb{S} \to \mathbb{S}.
\]
This ensures the following:

\begin{lem}\label{lem:3manifoldtimesI}
    Let $Y$ be an oriented 3-manifold with a smooth involution $\tau$.
    Then the following holds:  
    \begin{itemize}
        \item The set of isomorphism classes of real spin structures on $[0,1] \times Y$ bijectively corresponds to that on~$Y$.  
        \item The set of isomorphism classes of real spin$^c$ structures on $[0,1] \times Y$ bijectively corresponds to that on~$Y$.  
    \end{itemize}
\end{lem}

\begin{lem}
Suppose $Y$ is a closed, oriented 3-manifold equipped with a smooth involution  
\[
\tau \colon Y \to Y.
\]
Regarding existence, we have the following:  
\begin{itemize}
    \item If $\tau$ has a non-empty fixed point set of codimension $2$, then the condition  
    \[
    \tau^* \mathfrak{s} \cong \overline{\mathfrak{s}} \qquad \text{ and } \qquad H^1(Y; \Z)^{-\tau^* } =0 
    \]
    is sufficient to ensure the existence of a real spin$^c$ structure on $\mathfrak{s}$.
    
    \item If $\tau$ is free, then we have a regular double cover $Y \to Y/\tau$. In this case, the existence of an $O(2)$-bundle $E$ over $Y/ \tau$ such that 
    \begin{align*}
    w_2 (E) + w_1(E)^2  = 0 \qquad \text{ and }\qquad 
    w_1(E) = w_1 (l)
    \end{align*}
    is sufficient to ensure the existence of a real spin$^c$ structure on $Y$, where $l$ is the local system with fiber $\mathbb{Z}$ induced by the double cover.
\end{itemize}
For the classification, if we fix a spin$^c$ structure $\mathfrak{s}$ on $Y$, there is a one-to-one correspondence between isomorphism classes of real structures and  
\[
H^2(Y)^{-\tau^*} \oplus \frac{H^1(Y)^{\tau^*}}{\operatorname{im}(1+\tau^* \colon H^1(Y) \to H^1(Y))}
\]
in the non-free case. In the free case, the classification is given by  
\[
H^2(Y/\tau; l).
\]
\end{lem}

\begin{proof}
    By \cref{lem:3manifoldtimesI}, the proof follows by taking $Y$, forming the product $[0,1] \times Y$, and applying~\cref{lem:spinc4manifold}.
\end{proof}

Similarly, we obtain the following result for real spin structures:

\begin{lem}\label{classification:real_spiny}
Suppose $Y$ is a closed, oriented 3-manifold equipped with a smooth involution  
\[
\tau \colon Y \to Y.
\]
Suppose $\mathfrak{s}$ is a spin structure on $Y$. Regarding existence, we have
\begin{itemize}
    \item If $\tau$ has a non-empty fixed point set of codimension $2$, then the condition  
    \[
    \tau^* \mathfrak{s} \cong \mathfrak{s}
    \]
    is sufficient for $\mathfrak{s}$ to admit a real spin structure. In this case, $\tau$ is called an \emph{odd involution}.

    \item If $\tau$ is free, then we have a double cover $Y \to Y/\tau$. In this case, the condition
    \[
    w_1(l)^2 = 0,
    \]
    is a sufficient condition for $Y$ to adimt a real spin structure. Here, $l$ is the real line bundle  
    \[
    l = Y \times_{\mathbb{Z}_2} \mathbb{R} \to Y/\tau.
    \]
    In this case, $Y/\tau$ also admits a real spin structure.
\end{itemize}
For the classification, if a real spin structure exists, its classification is given by the isomorphism classes of real line bundles with an order $2$ lift of $\tau$ and the choice of the sign of $I$.   In particular, if we fix a spin structure $\mathfrak{s}$ on $Y$, there are exactly two choices of real spin structures on it. \qed
\end{lem}

\subsection{Real structures on surgeries}
Let $Y$ be an oriented $\mathbb{Z}_2$-homology $3$-sphere, and let $K$ be a null-homologous knot in $Y$. Then, there exists a unique double branched cover along $K$ because there is a unique non-trivial homomorphism $$H_1(Y \setminus K) \twoheadrightarrow \mathbb{Z}_2.$$ Let us denote the double branched cover by $\Sigma_2(Y, K)$ and let $\tau$ be the covering involution. In this situation, we have the following lemma:

\begin{lem}\label{classification_real_spinc}
    There exists unique real spin structure on  $\Sigma_2(Y, K)$ up to isomorphism and sign. 
\end{lem}

\begin{proof}
   Since $H^1(\Sigma_2(Y, K); \mathbb{Z}_2) = 0$, there exists a unique spin structure on $\Sigma_2(Y, K)$. The involution $\tau$ preserves this spin structure, and its fixed point set has codimension $2$. Therefore, by \Cref{classification:real_spiny}, we conclude that this spin structure admits a real spin structure.
   On the other hand, there is no non-trivial line bundle over $Y$. This means there is no non-trivial line bundle with involutions over $\Sigma_2(Y, K)$. Again from \cref{classification:real_spiny}, we see the uniqueness. 
   This completes the proof. 
\end{proof}   
 Fix an even integer $n$ and a null-homologous knot $K$ in an oriented 
 $\Z_2$-homology 3-sphere $Y$. Consider the trace of $n$-framed surgery $X_n(K)$, which can be viewed as a cobordism from $Y$ to the $n$-framed surgery $Y_n(K)$ along $K$. Let $S_K$ denote the core of the 2-handle of $X_n(K)$, regarded as a knot cobordism from $K$ to the empty set
\[
S_K \colon (Y, K ) \to (Y_n (K), \emptyset ) . 
\]
The following computation follows directly from the Mayer–Vietoris sequence:

\begin{lem} \label{homology} We have the following homological properties of $X_n(K)$: \begin{itemize} 

\item[(i)] $H_i(X_n(K)) \cong
\begin{cases} \mathbb{Z}, & \text{if } i = 0,2 \\ H_1(Y), & \text{if } i = 1 \\ 0, & \text{otherwise}. \end{cases}$ 

\item[(ii)] The intersection form of $X_n(K)$ is represented by $(n)$. 

\item[(iii)] $H_1(X_n(K) \setminus S_K) \cong \begin{cases} \mathbb{Z} \oplus H_1(Y), & \text{if } n = 0 \\ \Z_{|n|}\oplus H_1(Y) , &  \text{otherwise}. \end{cases}$ 

\item[(iv)] The inclusion-induced map $H_1(Y \setminus K) \twoheadrightarrow  H_1(X_n(K) \setminus S_K)$ is surjective. \qed

\end{itemize} \end{lem}

Since we assume that $n$ is even, there exists a unique non-trivial homomorphism
\[
H_1(X_n(K)\setminus S_K) \cong \mathbb{Z}_{|n|}\oplus H_1(Y) \twoheadrightarrow \Z_2. 
\]
Thus, we can consider the double branched cover of $X_n(K)$ along $S_K$, denoted by $\wt{X}_n(K)$, which can be viewed as a $\mathbb{Z}_2$-equivariant $4$-manifold cobordism from $\Sigma_2(Y, K)$ to $\wt{Y}_n(K)$. Here, $\wt{Y}_n(K)$ is the regular double cover associated with the unique non-trivial homomorphism
\[
H_1(Y_n(K)) \twoheadrightarrow \Z_2. 
\]
We collect the following computations for the double branched cover of $X_n(K)$.

\begin{lem}\label{lem:prop4surgery}
Suppose $n$ is non-positive and divisible by $4$. Then the following holds:
\begin{itemize} 

\item[(i)] $\wt{X}_n(K)$ is spin. 

\item[(ii)] $b^+(\wt{X}_n(K)) = 0$. 

\item[(iii)] $H^1(\wt{X}_n(K), \partial \wt{X}_n(K); \mathbb{Z}_2) = 0$. 

\end{itemize} \end{lem}

\begin{proof}
    Note that $\wt{X}_n(K)$ can be viewed as the $(n/2)$-framed surgery trace of $\wt{K} \subset \Sigma_2(Y, K)$, where $\wt{K}$ is a lift of $K$. Since $n/2$ is even, it follows from \Cref{homology} that $\wt{X}_n(K)$ is spin. Moreover, the assumption that $n$ is non-positive implies that $b^+(\wt{X}_n(K)) = 0$. The final bullet point follows from the Mayer–Vietoris sequence, similar to the computations in \Cref{homology}.
\end{proof}

\begin{cor}\label{uniquenessof0surgery}
    For any null-homologous knot $K$ in an oriented $\Z_2$-homology $3$-sphere $Y$, the double branched cover $\wt{X}_0(K)$ of the trace of zero-framed surgery $X_0(K)$ has a real spin structure. 
\end{cor}

\begin{proof}
From \cref{lem:prop4surgery} (i), we have that $\wt{X}_0(K)$ is spin.
Moreover, the fixed point set of the branched covering involution on $\wt{X}_0(K)$ is non-empty and has codimension two.
From \cref{classification:real_spin} (i), we obtain a real spin structure on $\wt{X}_0(K)$.  
\end{proof}

\section{Preliminaries of real Floer homotopy type for knots}

\subsection{Representations}\label{representations}

For a finite-dimensional vector space $V$, we denote by $V^+$ the one-point compactification of $V$.  We define the group $G = \mathbb{Z}_4$ as the cyclic group of order $4$ generated by $j \in \mathrm{Pin}(2)$, i.e.,  
\[
G = \{1, j, -1, -j\}.
\]
Define a subgroup $H$ of $G$ by  
\[
H = \{1, -1\} \subset G.
\]
Let $\mathbb{R}$ denote the trivial 1-dimensional real representation of $G$.  
Let $\tilde{\mathbb{R}}$ be the 1-dimensional real representation of $G$ defined via the surjection $G \to \mathbb{Z}_2 = \{1,-1\}$ and the corresponding scalar multiplication of $\mathbb{Z}_2$ on $\mathbb{R}$. Similarly, let $\tilde{\mathbb{C}}$ be the 1-dimensional complex representation of $G$ defined via the surjection $G \to \mathbb{Z}_2$ and the scalar multiplication of $\mathbb{Z}_2$ on $\mathbb{C}$.  
We introduce also $G$-representations $\mathbb{C}_+$ and $\mathbb{C}_-$
 where $\mathbb{C}_+$ and $\mathbb{C}_-$ are complex 1-dimensional representations defined by assigning to $j \in G$ multiplication by $i$ and $-i$, respectively.  
Note that we have the relation $\C_- = \tilde{\C} \cdot \C_+$
in the complex representation ring $R(G)$ of $G$. It is straightforward to check that the representation ring $R(G)$ is given by
\[
R(G) = \Z[w,z]/(w^2-2w, w-2z+z^2).
\]
Here, the generators $w$ and $z$ are given as the $K$-theoretic Euler classes of $\tilde{\mathbb{C}}$ and $\mathbb{C}_+$, namely,  
\[
w = 1 - \tilde{\mathbb{C}} \qquad \text{ and } \qquad z = 1 - \mathbb{C}_+.
\]
The augmentation map is given by  
\[
R(G) \to \mathbb{Z} ; \qquad
w, z \mapsto 0,
\]
so the augmentation ideal is given by $(w, z) \subset R(G)$. Readers may compare this expression with the standard expression of $R(G)$, given by  
\[
R(G) = \mathbb{Z}[t] / (t^4 - 1).
\]
An isomorphism is given as follows:
\[
\mathbb{Z}[w, z] / (w^2 - 2w, w - 2z + z^2) \to \mathbb{Z}[t] / (t^4 - 1) ; \qquad w \mapsto 1 - t^2, \qquad z \mapsto 1 - t.
\]

\subsection{Construction of real Floer homotopy type}
In this section, we review the construction of real Floer homotopy types for knots. For details, see \cite{Ma03, KMT:2021, KMT:2024}.

Although not essential, we focus on the case $Y = S^3$ in this section. 
Let $\mathfrak{s}_0$ be the unique spin structure on the double branched cover $\Sigma_2(K)$, $\tau \colon \Sigma_2(K) \to \Sigma_2(K)$ be the deck transformation, and $P$ be the principal $\operatorname{Spin}(3)$ bundle for $\mathfrak{s}_0$. Since the fixed point set is codimension $2$, we can take an order $4$ lift $\widetilde{\tau} \colon P \to P$ of the induced map
\[
\tau_* \colon {\rm SO}(T\Sigma_2(K)) \to {\rm SO}(T\Sigma_2(K)),
\]
where $SO(T\Sigma_2(K))$ is the orthonormal frame bundle of $\Sigma_2(K)$ with respect to a fixed invariant metric $g$ on $\Sigma_2(K)$.  Then, we have the infinite-dimensional functional
\[
CSD \colon \mathcal{C}_K := \left(i\ker d^* \subset i\Omega^1_{\Sigma_2(K)} \right) \oplus \Gamma (\mathbb{S}) \to \mathbb{R},
\]
called the \emph{Chern--Simons--Dirac functional}, where $\mathbb{S}$ is the spinor bundle with respect to $\mathfrak{s}_0$, $\Gamma (\mathbb{S})$ denotes the set of sections of $\mathbb{S}$, and $d^*$ denotes the $L^2$-formal adjoint of $d \colon i\Omega^0_{\Sigma_2(K)} \to i\Omega^1_{\Sigma_2(K)}$.   The Seiberg--Witten Floer homotopy type is defined as the Conley index of the finite-dimensional approximation of the formal gradient flow of $CSD$. For that purpose, we describe the formal gradient of $CSD$ as the sum $l+c$, where $l$ is a self-adjoint elliptic part and $c$ is a compact map. Then, we decompose $\mathcal{C}_K$ into eigenspaces of $l$. 
Define $V^{\lambda}_{-\lambda}(K) \oplus W^{\lambda}_{-\lambda}(K)$ to be the direct sums of the eigenspaces of $l$ whose eigenvalues are in $(-\lambda, \lambda]$ for form and spinor parts respectively and restrict the formal gradient flow $l + c$ to $V^{\lambda}_{-\lambda}(K)\oplus W^{\lambda}_{-\lambda}(K) $ for a fixed $\lambda \gg 0$.  Then by considering the Conley index $(N, L)$ for $\left(V^{\lambda}_{-\lambda}(K) \oplus W^{\lambda}_{-\lambda}(K) , l+p^{\lambda}_{-\lambda}c\right)$ with a  certain cutting off, we get the Manolescu's Seiberg--Witten Floer homotopy type
\begin{align*}
SWF(\Sigma_2(K), \frak{s}_0) := \Sigma^{-V^0_{-\lambda} \oplus W^0_{-\lambda}-n(\Sigma_2(K), \fraks_0, g))\C } N / L,  
\end{align*}
where $n(Y, \fraks_0, g)$ is the quantity given in \cite{Ma03} and $g$ is a Riemann metric on $\Sigma_2(K)$. For the meaning of desuspensions and how to formulate a well-defined homotopy type in a certain category, see \cite{Ma03}. 
For the latter purpose, we take $g$ as $\Z_2$-invariant metric.
Since we are working with the spin structure $\frak{s}_0$, we have an additional $\operatorname{Pin}(2)$-action on the configuration space $\mathcal{C}_K$ which preserves the values of $CSD$.
Now, we define an involution on $\mathcal{C}_K$
\[
I := j \circ \wt{\tau},  
\]
where $j$ is the quaternionic element in $\operatorname{Pin}(2) = S^1 \cup j\cdot S^1$, which represents the real spin structure on $\Sigma_2(K)$. \footnote{The map $-I$ also induces another real involution on the configuration space. One can easily check the invariants $\delta_R, \underline{\delta}_R$ and $\overline{\delta}_R$ we will focus on in this paper do not depend on such choices. }

Since $I$ also acts on $\mathbb{S}$ \emph{anti-complex linearly}, the lift $I$ is called a \emph{real structure} on $\mathfrak{s}_0$. Combined with the ${\rm Pin(2)}$-action of the Seiberg--Witten Floer homotopy type, we can take the \emph{Conley index} so that we have a ${\rm Pin(2)} \times_{\pm 1} \mathbb{Z}_4$-action on $SWF(\Sigma_2(K), \mathfrak{s}_0)$. 
We use the $j$-action on $\Gamma(\mathbb{S})^I$ to get a complex structure on it, which is modeled as $\C_+^\infty$. 

Recall that $G=\langle j \rangle$ and $H= \langle -1 \rangle \subset G$.  Now, we define 
\begin{align*}
SWF_R(K) :&= \Sigma^{-(V^0_{-\lambda} \oplus W^0_{-\lambda} )^I - \frac{1}{2} n (\Sigma_2(K), \mathfrak{s}_0, g) \mathbb{C}_+ } N^I / L^I \\
& =  \left[\left(N^I/L^I,
\dim_{\mathbb{R}}\left(V^0_{-\lambda}\right)^I,
\dim_{\mathbb{C}}\left(W^0_{-\lambda}\right)^I
+ \frac{n(Y, \mathfrak{t}, g)}{2} \right)\right] \in \mathfrak{C}_{G}
\end{align*}
which we call the \emph{real Seiberg--Witten Floer homotopy type} for $K$. Here, we take a \emph{${\rm Pin(2)} \times_{\pm 1} \mathbb{Z}_4$-invariant index pair} $(N, L)$ for the flow 
$\left(V^{\lambda}_{-\lambda}(K) \oplus W^{\lambda}_{-\lambda}(K), l + p^{\lambda}_{-\lambda} c\right)$
with a certain cutting-off. Since the action of $j$ commutes with $I$, we obtain a \emph{$G$-action} on the stable homotopy types $SWF_R(K)$.

Here, we briefly explain the category $\mathfrak{C}_{G}$ containing our Floer homotopy type. The objects of this category are tuples $(X, m, n)$, where:
\begin{itemize}
    \item $X$ is a pointed finite $G$-CW complex satisfying the following properties:
    \begin{itemize}
        \item $X^{H}$ is $G$-homotopy equivalent to $(\widetilde{\mathbb{R}}^s)^+$ for some $s \geq 0$.
        \item $G$ acts freely on $X \setminus X^{H}$.
    \end{itemize}
    \item $m$ is an integer. \item $n$ is a rational number. 
\end{itemize}
The set of morphisms from $(X, m, n)$ to $(X', m', n')$ is given by
\[
\operatorname{colim}_{l \to \infty} \left[ (\widetilde{\mathbb{R}}^{m+l} \oplus \mathbb{C}^{n+l}_+ )^+ \wedge X , (\widetilde{\mathbb{R}}^{m'+l} \oplus \mathbb{C}^{n'+l}_+  )^+ \wedge X' \right]_{G}.
\]
The Floer homotopy type determines a well-defined element
\[
[SWF_R(K)] \in \mathfrak{C}_{G}.
\]
Therefore, we have the following two equivariant cohomologies:
\begin{align*}
\widetilde{H}^*_{\mathbb{Z}_2} (SWF_R(K); \mathbb{F}_2) := \widetilde{H}^{* + \dim (V^0_{-\lambda})^I + 2 \dim_{\mathbb{C}}(W^0_{-\lambda})^I  +  n (\Sigma_2(K), \mathfrak{s}_0, g) }_{\mathbb{Z}_2} (N^I/L^I; \mathbb{F}_2),
\end{align*}
and
\begin{align*}
\widetilde{H}^*_{\mathbb{Z}_4} (SWF_R(K); \mathbb{F}_2) := \widetilde{H}^{* + \dim (V^0_{-\lambda})^I + 2 \dim_{\mathbb{C}}(W^0_{-\lambda})^I  +  n (\Sigma_2(K), \mathfrak{s}_0, g) }_{\mathbb{Z}_4} (N^I/L^I; \mathbb{F}_2),
\end{align*}
which are modules over $H^*(B\mathbb{Z}_2; \mathbb{F}_2)$ and $H^*(B\mathbb{Z}_4; \mathbb{F}_2)$, respectively.


\begin{rem}
    In the above construction of the Floer homotopy type $SWF_R(K)$, we first take a finite-dimensional approximation of the gradient flows of $CSD$ and the Conley index, and then we take the \emph{$I$-invariant} part. As an alternative construction, we can first take the \emph{$I$-invariant} of the configuration space and then compute the \emph{$G$-equivariant Conley index}. One can see that these two constructions are equivalent, which follows from formal properties of equivariant Conley indices.  In a later section, when we define \emph{real Floer homotopy types} for homology $S^1 \times S^2$, we shall first take the $I$-invariant part of the configuration space and then apply equivariant Conley index theory.
\end{rem}

Next, we consider the local equivalence classes which leads to the concordance invariants:
\begin{defn}
We say that two objects $(X, m, n)$ and $(X', m', n')$ are \emph{locally equivalent} if there exist $G$-equivariant maps
\begin{align*}
f &\colon (\widetilde{\mathbb{R}}^{m + \ell} \oplus \mathbb{C}_+^{n + \ell})^+ \wedge X \to (\widetilde{\mathbb{R}}^{m' + \ell} \oplus \mathbb{C}_+^{n' + \ell})^+ \wedge X', \\
f' &\colon (\widetilde{\mathbb{R}}^{m' + \ell} \oplus \mathbb{C}_+^{n' + \ell})^+ \wedge X' \to (\widetilde{\mathbb{R}}^{m + \ell} \oplus \mathbb{C}_+^{n + \ell})^+ \wedge X
\end{align*}
for sufficiently large $\ell \gg 0$, which induce $G$-homotopy equivalences on the $H$-fixed point sets.
\end{defn}

We denote by $\mathcal{LE}$ the set of all objects $[(X, m, n)]_{\mathrm{loc}}$ up to local equivalence. We regard $\mathcal{LE}$ as a group, with the group operation defined by
\[
[(X, m, n)]_{\mathrm{loc}} + [(X', m', n')]_{\mathrm{loc}} := [(X \wedge X', m + m', n + n')]_{\mathrm{loc}}.
\]
The element $[(S^0, 0, 0)]_{\mathrm{loc}}$ serves as the identity element, and the Whitehead--Spanier dual with $(-m, -n)$ provides the inverse. With respect to this group structure, the following result is proven in \cite{KMT:2024}:

\begin{thm}
    The map 
    \[
     \mathcal{C} \to \mathcal{LE} ; \qquad [K] \mapsto [SWF_R(K)]_{\text{loc}}
    \]
    defines a group homomorphism, where $\mathcal{C}$ denotes the smooth knot concordance group. 
\end{thm}

The invariants $\delta_R$, $\underline{\delta}_R$, and $\overline{\delta}_R$ are defined via factorization through $\mathcal{LE}$, meaning that we have functions:
\[
\delta_R, \underline{\delta}_R, \overline{\delta}_R \colon  \mathcal{LE} \to \frac{1}{16} \mathbb{Z}.
\]
See \cite{KMT:2024} for more details. 

To define the real kappa invariant $\kappa_R$, we perform the \emph{doubling construction} in \cite[Section 3.3]{KMT:2021}, which can be described as the functor:
\[
\mathfrak{C}_{G} \to \mathfrak{D}_{G}, 
\]
where $\mathfrak{D}_{G}$ is a category of spectrum whose objects are double in the sense of \cref{double} and morphisms are natural doubled morphisms described in \cite[Section 3.3]{KMT:2021}. 
More detailed explanations are given in \cref{real 10/8 inequalities}. 
The invariant $\kappa_R$ is defined using \emph{$G$-equivariant K-theoretic information} for objects in $\mathfrak{D}_{G}$. See \cite{KMT:2021} for further details.

\section{${\rm Pin}^-(2)$-Monopoles for Homology $S^1 \times S^2$}

In this section, we introduce a \emph{real Floer homotopy type}
\[
SWF_R (Y) \in \mathfrak{C}_{G}
\]
for a given oriented homology $S^1 \times S^2$ manifold $Y$.

\subsection{Real spin structures for homology $S^1 \times S^2$}

Let $Y$ be a homology $S^1 \times S^2$. Then, there exists a unique double cover $\widetilde{Y} \to Y$ determined by the unique surjection
\[
\pi_1(Y) \twoheadrightarrow \mathbb{Z}_2.
\]
Let $\mathfrak{s}_0$ be a spin structure on $Y$, and let $\widetilde{\mathfrak{s}}_0$ be its pullback spin structure on $\widetilde{Y}$. Denote by $\tau \colon \widetilde{Y} \to \widetilde{Y}$ the covering involution.


\begin{lem}
    We have $H^1(\widetilde{Y}; \mathbb{Z}_2) = \mathbb{Z}_2$. In particular, there are exactly two isomorphism classes of real line bundles over $\widetilde{Y}$.
\end{lem}

\begin{proof} 
Recall that the double cover $p \colon \wt{Y} \to Y$ induces a long exact sequence in cohomology, where $\text{tr}^*$ is the transfer map:
\[
\cdots \to H^i(Y; \Z_2) \xrightarrow{p^*} H^i(\wt{Y}; \Z_2) \xrightarrow{\text{tr}^*} H^i(Y; \Z_2) \to H^{i+1}(Y; \Z_2) \to \cdots.
\]
Since \( H^1(\wt{Y}; \Z_2 ) \neq 0 \), this implies in particular that 
\[
H^1(\wt{Y}; \Z_2 ) \xrightarrow{\text{tr}^*} H^1(Y; \Z_2 ) = \Z_2
\qquad \text{ and } \qquad
H^2(Y; \Z_2 ) = \Z_2 \xrightarrow{p^*} H^2(\wt{Y}; \Z_2 ) 
\]
are isomorphisms, which completes the proof.
\end{proof}

\begin{prop}\label{s1s2}
    The 3-manifold $\widetilde{Y}$ with the covering involution $\tau \colon \widetilde{Y} \to \widetilde{Y}$ admits a real spin structure. Moreover, the real spin structure is unique up to isomorphism and sign.
\end{prop}

\begin{proof}
One can see the unique non-trivial local system $l \colon  \pi_1(Y) \twoheadrightarrow \Z_2$ satisfies 
    \[
     w_1^2 (l) = 0. 
    \]
    This follows from the following general argument: 
Suppose $Y$ is a homology $S^1 \times S^2$. Then, the generator of $H_2(Y)$ is represented by a smoothly embedded oriented surface $F$ in $Y$. Take a normal neighborhood $\nu$ of $F$ in $Y$, which is a trivial bundle over $F$. Thus, we have a map from $\nu \cong F \times I$ to $S^1$, where each $\{x\} \times I$ wraps around $S^1$ exactly once. 
Now, extend this map to all of $Y$ by sending every point outside of $\nu$ to $1 \in S^1$. The induced map $Y \to S^1$ induces an isomorphism on the first cohomology 
$H^1(S^1) \cong H^1(Y)$.
This shows that $a^2 = 0$ for the generator $a \in H^1(Y)$.

Then, from \cref{classification:real_spiny}, we obtain a real spin structure on $\widetilde{Y}$.  Again, from \cref{classification:real_spiny}, the isomorphism classes of real spin structures correspond to real line bundles on $\widetilde{Y}$ with involutive lifts of the involution. However, this construction induces a line bundle on $Y$, and the pullbacks of all line bundles on $Y$ to $\widetilde{Y}$ are trivial. This completes the proof.
\end{proof}

\subsection{Real Floer homotopy type for homology $S^1 \times S^2$}

In this section, we introduce a real Seiberg--Witten Floer homotopy type for homology $S^1 \times S^2$ manifolds. 

Let $Y$ be a homology $S^1 \times S^2$, and let $l$ be the unique non-trivial $\mathbb{Z}_2$-bundle over $Y$. Take a Riemannian metric $g$ on $Y$ and pull it back to $\widetilde{Y}$, denoted by $\widetilde{g}$.  We choose the spin structure $\mathfrak{s}$ on $\widetilde{Y}$ such that the natural lift $\widetilde{\tau}$ of the covering involution has order 4. Let $\mathbb{S}$ denote the spinor bundle for $\mathfrak{s}$. Then, we have an antilinear involution
\[
I \colon \mathbb{S} \to \mathbb{S}
\]
as a real spin structure, which acts on $\Lambda^1_{\widetilde{Y}}(i \mathbb{R})$ by $-\tau^*$. 
Let $A_0$ denote the spin connection with respect to $\widetilde{g}$. This defines an $I$-invariant Chern--Simons--Dirac functional on $\widetilde{Y}$:
\[
CSD \colon C(\widetilde{Y}) =  (A_0 + \Omega^1_{\widetilde{Y}}(i\mathbb{R})) \oplus \Gamma (\mathbb{S}) \to \mathbb{R},  
\]
where $A_0$ is the spin connection. 
We consider the fixed point set with respect to $I$:
\[
CSD \colon C^I(\widetilde{Y}) =  (A_0+ \Omega^1_{\widetilde{Y}}(i\R ))^I \oplus \Gamma^I (\mathbb{S}) \to \mathbb{R}.
\]

The gauge group in this setting is written as 
\[
\mathcal{G}^I(\widetilde{Y}) := \{ g \in  \operatorname{Map}(\widetilde{Y}, U(1) ) \mid \tau^* g = \overline{g} \}.
\]
In this situation, the global slice of $C^I(\widetilde{Y})$ with respect to the action of $\mathcal{G}^I(\widetilde{Y})$ is given by 
\[
V^I (\widetilde{Y}) = (A_0 + \operatorname{Ker} d^*)^{-\tau_* } \oplus \Gamma^I(\mathbb{S})
\]
with a $G = \langle j \rangle$-action.  We take a finite-dimensional approximation 
\[
V^\lambda_\mu \oplus W^\lambda_\mu \subset V^I(\widetilde{Y})
\]
by the sums of certain eigenspaces for $\lambda, -\mu \gg 0$.

Now, we claim a compactness result for the flow lines.

\begin{lem}
    The gradient flow for the restricted vector field  
    \[
    l + p^\lambda_\mu \circ c \colon V^\lambda_\mu \oplus W^\lambda_\mu \to T(V^\lambda_\mu \oplus W^\lambda_\mu)
    \]
    of $CSD$ has a large ball as an isolating neighborhood.
\end{lem}

 \begin{proof}
    Originally, when defining the Seiberg--Witten Floer homotopy type for homology $S^1 \times S^2$ manifolds, we need to carefully consider the action of the subgroup of the gauge group
    \[
  S^1\times \mathcal{G}_0(Y) \times \Z  = S^1 \times \mathcal{G}_0(Y) \times H^1(Y; \mathbb{Z}) \subset \mathcal{G}_Y
    \]
    to ensure the usual compactness of the critical point sets and flow lines, where $\mathcal{G}_0(Y)$ denotes the contractible identity component of the gauge group. 

    In our case, since we are considering the restricted $I$-invariant gauge group
    \[
    \mathcal{G}^I(\widetilde{Y}) = \{ g \in \operatorname{Map} (\widetilde{Y}, U(1)) \mid \tau^* g = \overline{g} \}
    \]
    which is homotopy equivalent to the constant gauge group $\mathbb{Z}_2$, we can ignore the action of such a non-compact group after taking the $I$-invariant part. 

    For this reason, on the $I$-fixed point set, we obtain the same compactness properties for the critical point set and flow lines as in the case of Seiberg--Witten Floer homotopy for rational homology 3-spheres.  Thus, the lemma follows from essentially the same argument as in \cite{Ma03}.
\end{proof}

 Recall that $G= \langle j \rangle$. 
Now, we follow Manolescu's construction to obtain the Seiberg--Witten Floer homotopy type and define the $G$-equivariant Conley index
\[
I^{\lambda}_\mu (\widetilde{Y}, g)
\]
for $\lambda$ and $\mu$. 
Then, we define
\begin{align*}
SWF_R(Y) &:= \Sigma^{- (V^0_\mu)^I - (W^0_\mu)^I } \Sigma^{\C_+^{n(\widetilde{Y}, g)}} I^{\lambda}_\mu (\widetilde{Y}, g) \\
&= \left[\left(N^I/L^I,\,
\dim_{\mathbb{R}}(V^0_{-\lambda})^I,\,
\dim_{\mathbb{C}}(W^0_{-\lambda})^I + \frac{n(\widetilde{Y}, \mathfrak{t}, g)}{2} \right)\right] \in \mathfrak{C}_{G},
\end{align*}
where the correction term $n(\widetilde{Y}, g)$ is defined by
\[
n(\widetilde{Y}, g) := \operatorname{ind}_{\mathbb{C}}(D_{\widetilde{g}}) + \frac{c_1(\mathfrak{s})^2 - \sigma(X)}{8},
\]
and $X$ is any real spin 4-manifold bounding $\widetilde{Y}$ such that there exists a real spin structure $\mathfrak{s}'$ on $X$ that restricts to the given real spin structure on $\widetilde{Y}$. 

Let $\widetilde{g}$ be a $\mathbb{Z}/2$-equivariant metric on $X$ that coincides with $g$ on $\widetilde{Y}$. Let $D_{\widetilde{g}}$ be the Dirac operator on $\mathfrak{s}'$ with respect to the metric $\widetilde{g}$. We can construct such an $X$ for any $\widetilde{Y}$ as follows: let $l$ be a loop in $\widetilde{Y}$ representing the generator of $H_1(Y, \mathbb{Z}) \cong \mathbb{Z}$. Define the 4-manifold
\[
X:= [0,1] \times \widetilde{Y} \cup_{\{1\} \times \nu(l) = D^2 \times \partial \widetilde{D}^2} D^2 \times \widetilde{D}^2,
\]
where $\widetilde{D}^2$ is the 2-disc $\{z \in \mathbb{C} \mid \lvert z \rvert \leq 1\}$ equipped with the involution $z \mapsto -z$. One can easily check that $X$ satisfies the conditions above. Since the $\mathbb{R}$-spectral flow of a family of real Dirac operators is equal to half the $\mathbb{R}$-spectral flow of the corresponding family of the usual Dirac operators, the class $SWF_R(Y)$ is independent of the choice of metric, as an element of $\mathfrak{C}_G$. This, in particular, implies the following:

\begin{prop}
The $G$-equivariant stable homotopy class of $SWF_R(Y)$ is independent of the choices of $g$, $\lambda$, and $\mu$. \qed
\end{prop}

\begin{rem}
    Although we focus on homology $S^1 \times S^2$ manifolds, our real Floer homotopy type can be defined for any real spin$^c$ or spin 3-manifold $(Y, \tau, \mathfrak{s}, I)$ satisfying
\[
b_1(Y) - b_1(Y/ \tau) = \dim H^1(Y; \mathbb{Q})^{-\tau_*} = 0.
\]
When we have a real spin structure, the corresponding Floer homotopy type
\[
SWF_R (Y, \tau, \mathfrak{s}, I)
\]
lies in $\mathfrak{C}_{G}$. In contrast, for a real spin$^c$ structure, we only obtain a well-defined $\mathbb{Z}_2$-action on the spectrum.
\end{rem}

\subsection{Cobordism maps between homology $S^1\times S^2$ manifolds}

Let $Y$ and $Y'$ be homology $S^1 \times S^2$ manifolds. Let $W$ be an oriented connected cobordism from $Y$ to $Y'$, and let $l$ be a real line bundle satisfying the following properties:  
\begin{itemize}
    \item[(i)] $l|_{Y}$ and $l|_{Y'}$ are non-trivial real line bundles.  
    \item[(ii)] $w_2(W) + w_1(l)^2 = 0$.
\end{itemize}

\begin{prop}
    For the cobordism $(W, l)$, there exists a $Spin^-(4)$ structure associated with $l$. In particular, the double cover $\widetilde{W}$ corresponding to $l$ admits a real spin structure.
\end{prop}

\begin{proof}
    This corresponds to the boundary version of \cref{classification:real_spin}. The original proof by Nakamura \cite{Na13}, which uses obstruction theory, remains valid in the case of 4-manifolds with boundary. Therefore, the condition $w_2(W) + w_1(l)^2 = 0$ is sufficient to obtain a $Spin^-(4)$ structure associated with $l$.
\end{proof}

Moreover, we also suppose 
\[
b_1(W)- b_1(W/\tau) =0
\]
for simplicity in this section.  For this cobordism $(W, l)$, we claim that it induces the following cobordism map:

\begin{thm}\label{realBF}
    There exists a $G$-equivariant map
    \[
    BF^R_W \colon SWF_R(Y) \wedge (\widetilde{\mathbb{R}}^{\frac{- \sigma(W)}{8}})^+ \to SWF_R(Y') \wedge (\mathbb{R}^{b^+(W; l)})^+.
    \]
    Here, $b^+(W; l)$ denotes the dimension of a maximal positive definite subspace of $H^2(W; l)$ with respect to the intersection form of the local coefficient cohomology:
    \[
    H^2(W; l) \otimes H^2(W; l) \to \mathbb{R}.
    \]
    The quantity $\sigma(W)$ refers to the signature of the intersection form restricted to the image of
    \[
    \operatorname{im}(H^2(W, \partial W) \to H^2(W)).
    \]
    In particular, if $b^+(W; l) = 0$, we obtain a local map from $SWF_R(Y)$ to $SWF_R(Y')$.
\end{thm}

\begin{proof}
    We give a sketch of the proof. The proof is analogous to Nakamura's construction \cite{Na13} and the construction of the Bauer--Furuta map described in \cite{KMT:2024, Mi23}.

Let $\mathbb{S}^{\pm}$ denote the positive and negative spinor bundles on $\widetilde{W}$, and let $\mathbb{S}$ and $\mathbb{S}'$ be the spinor bundles on $\widetilde{Y}$ and $\widetilde{Y}'$, where $\widetilde{Y}$ and $\widetilde{Y}'$ are the double covers of $Y$ and $Y'$ associated with $l|_Y$ and $l|_{Y'}$, respectively.  We take a real spin structure
\[
I \colon \mathbb{S}^{\pm} \to  \mathbb{S}^{\pm}.
\]
Here, we consider the Sobolev norms $L^2_k$ for the spaces $\Omega^\ast(W)$ and $\Gamma(\mathbb{S}^{\pm})$, obtained from $\tau$-invariant metrics and $\tau$-invariant connections, for a fixed integer $k \geq 3$. We define the Seiberg--Witten map on $\widetilde{W}$ as a finite-dimensional approximation of the map
\begin{align}
    \label{eq: SW map}
    SW \colon \Omega_{CC}^1(W) \times \Gamma(\mathbb{S}^+)
    \to \Omega^+(W) \times \Gamma(\mathbb{S}^-) \times \hat{V}(-Y)^{\mu}_{-\infty} \times \hat{V}(Y')^{\mu}_{-\infty}
\end{align}
for large $\mu$. The following gives precise definitions of the notations appeared in \eqref{eq: SW map}. 
\begin{itemize}
\item 
The notation 
\[
\Omega_{CC}^1(W) 
=\Set{ a \in \Omega^1(W) \mid d^{\ast} a=0, d^{\ast} \mathbf{t} a=0, \int_{Y} \mathbf{t} \ast a=0, \int_{Y'} \mathbf{t} \ast a=0 }
\]
is the space of $1$-forms satisfying the double Coulomb condition, introduced by Khandhawit in \cite{Khan15}, where $*$ denotes the Hodge star operator and $\mathbf{t}$ is the restriction as a differential form. 
\item For a general rational homology 3-sphere $Y$ equipped with a spin$^c$ structure, $\hat{V}(Y, \mathfrak{t})^\mu_{-\infty}$ is a subspace of 
\[
\hat{V}(Y, \mathfrak{t}) = \operatorname{Ker} d^* \times \Gamma(\mathbb{S})
\]
defined as the direct sum of eigenspaces whose eigenvalues are less than $\mu$ with respect to the operator $l= (*d, D_{A_0} )$ for the spin connection $A_0$.
\end{itemize}
The $\Omega^+(W) \times \Gamma(\mathbb{S}^-)$ factor of the map $SW$ corresponds to the Seiberg--Witten equations, while the $\hat{V}(-Y)^{\mu}_{-\infty} \times \hat{V}(Y')^{\mu}_{-\infty}$ factor is, roughly, the restriction of 4-dimensional configurations to 3-dimensional ones. 
One can verify this map SW is $I$-equivariant. 
As in the case of 3-manifolds, we first take the $I$-invariant part of the Seiberg--Witten equation:
\begin{align}
    \label{eq: SW map I}
    SW^I  \colon (\Omega_{CC}^1(W) \times \Gamma(\mathbb{S}^+))^I 
    \to (\Omega^+(W) \times  \Gamma(\mathbb{S}^-) \times \hat{V}(-Y)^{\mu}_{-\infty} \times \hat{V}(Y')^{\mu}_{-\infty} )^I.
\end{align}
Again, one can check a certain compactness to obtain a finite-dimensional approximation of the map \eqref{eq: SW map I} as in the usual Seiberg--Witten theory \cite{Ma03, Khan15}.  Taking a finite-dimensional approximation of this,
we obtain a $G$-equivariant map of the form
\begin{align*}
f \colon \Sigma^{m_0\tilde{\R}}\Sigma^{n_0 \mathbb{C}_+} I_{-\mu}^{-\lambda} (Y) \to \Sigma^{m_1\tilde{\R}}\Sigma^{n_1 \mathbb{C}_+ }I^\mu_\lambda (Y'),
\end{align*}
where $m_i, n_i^\pm \geq 0$ and $-\lambda, \mu$ are sufficiently large. Then, $m_0 - m_1$ and $n_0 - n_1$ can be computed as the indices of certain operators with APS boundary conditions, given by:
\begin{align*}
m_0 - m_1 ={}& \dim_\mathbb{R}\left(V(Y)^0_{\lambda}\right) - \dim_\mathbb{R}\left(V(Y')^0_{-\mu}\right) - b^+(W) + b^+_\tau(W), \\
\begin{split}
n_0 - n_1 ={}& \dim_\mathbb{C}\left(W(Y)^0_{\lambda}\right) - \dim_\mathbb{C}\left(W(Y')^0_{-\mu}\right) \\
&- \frac{1}{2} \left( \frac{c_1(\mathfrak{s})^2 - \sigma(W)}{8} + n(Y', \mathfrak{t}', g') - n(Y, \mathfrak{t}, g) \right), 
\end{split}
\end{align*}
where the vector spaces $V(Y)^0_{\lambda}$, $W(Y)^0_{\lambda}$, $V(Y')^0_{\lambda}$, and $W(Y')^0_{\lambda}$ denote the $I$-invariant parts of 1-forms and spinors on the global slices:
\begin{align*}
    V(Y)^0_{\lambda} \oplus W(Y)^0_{\lambda} &= \left(\hat{V}(Y)^0_{\lambda}\right)^I, \\
    V(Y')^0_{\lambda} \oplus W(Y')^0_{\lambda} &= \left(\hat{V}(Y')^0_{\lambda}\right)^I.
\end{align*}This completes the proof.
\end{proof}

Based on this, we introduce the following class of cobordisms:
\begin{defn}
Let $Y$ and $Y'$ be oriented homology $S^1 \times S^2$ manifolds. We say that $Y$ and $Y'$ are \emph{real homology cobordant} if there exists a smooth $\mathbb{Z}_2$-homology cobordism $W$ from $Y$ to $Y'$ such that the regular non-trivial double cover $\widetilde{W}$ of $W$ satisfies
\[
b_2(\widetilde{W}, \widetilde{Y}) = b_2(\widetilde{W}, \widetilde{Y}') = 0,
\]
where $\widetilde{Y}$ and $\widetilde{Y}'$ denote the regular non-trivial double covers of $Y$ and $Y'$, respectively.
\end{defn}

The local equivalence classes of homology $S^1 \times S^2$ manifolds depend only on their real homology cobordism classes:
\begin{prop}\label{R-cob}
    If $Y$ and $Y'$ are real homology cobordant, then 
    $SWF_R(Y)$ and $SWF_R(Y')$ are locally equivalent.
\end{prop}

\begin{proof}
    Since $W$ is a $\mathbb{Z}_2$-homology cobordism, the non-trivial line bundles on $Y$ and $Y'$ extend uniquely to $W$, which we denote by $l$. On the boundaries $Y$ and $Y'$, we have real spin structures on $\widetilde{Y}$ and $\widetilde{Y}'$ since 
    \[
    w_2 (Y)+w_1(l|_Y)^2   = 0 \qquad \text{ and } \qquad  w_2 (Y')+ w_1(l|_{Y'})^2= 0.
    \]
    We claim that these real spin structures extend to $\widetilde{W}$. This follows from the condition 
    \[
    w_2(W) + w_1(l)^2 = 0.
    \]
    Thus, we obtain a $G$-equivariant real Bauer--Furuta invariant:
    \[
    BF^R_W \colon SWF_R(Y) \wedge (\widetilde{\mathbb{R}}^{\frac{- \sigma(\widetilde{W})}{8}})^+ \to SWF_R(Y') \wedge (\mathbb{R}^{b^+(W; l)})^+.
    \]
    Here, we assume that $\sigma(\widetilde{W}) = 0$ and 
    \[
    b^+(W; l) = b^+(\widetilde{W}) - b^+_\tau (\widetilde{W}) = 0 - 0 = 0.
    \]
    Thus, this gives a local map from $SWF_R(Y)$ to $SWF_R(Y')$. The opposite direction of a local map is obtained by considering the real Bauer--Furuta invariant for $-\widetilde{W}$. This completes the proof.
\end{proof}


\subsection{Results on zero-framed surgery}

In this section, we prove \cref{0surgerydetermine}, which follows from the following general theorem:

\begin{thm}
    For any knot $K$ in an oriented homology 3-sphere $Y$, we have 
    \[
    [SWF_R(K)]_{\text{loc}} = [SWF_R(Y_0(K))]_{\text{loc}}.
    \]
\end{thm}

\begin{proof}
From \cref{uniquenessof0surgery}, we see that there exists a $\mathbb{Z}_2$-equivariant cobordism
\[
\widetilde{X}_0(K) \colon \Sigma_2(Y, K) \to \widetilde{S^3_0}(K),
\]
where $\widetilde{S^3_0}(K)$ denotes the unique non-trivial double cover of $S^3_0(K)$, obtained as the double branched cover along the 2-handle core of the trace of the zero-framed surgery on $K$. Again by \cref{uniquenessof0surgery}, there exists a unique real spin structure on this cobordism. One can readily verify that
\begin{align*}
\sigma(\widetilde{X}_0(K)) &= 0, \\
b_2(\widetilde{X}_0(K)) - b_2^\tau(\widetilde{X}_0(K)) &= 0.
\end{align*}

Using an argument similar to that in \cref{realBF}, we obtain a real cobordism map
\[
BF^R_{\widetilde{X}_0(K)} \colon SWF_R(K) \to SWF_R(S^3_0(K)),
\]
which is a local map. The reverse direction is given by considering the cobordism $-\widetilde{X}_0(K)$. This completes the proof.
\end{proof}

As a corollary, we immediately obtain \cref{0surgerydetermine}, whose statement we recall:

\begin{cor}
    For any knot $K$ in $S^3$, 
    the local equivalence class $[SWF_R(K) ]_{\text{loc}}$ only depends on the orientation preserving diffeomorphism type of $S^3_0(K)$. In particular, the concordance invariants 
    $$
    \delta_R(K), \quad \overline{\delta}_R(K), \quad \underline{\delta}_R(K), \quad \text{and} \quad \kappa_R(K)
    $$
    depend only on the orientation preserving diffeomorphism type of $S^3_0(K)$. \qed
\end{cor}

\noindent Note that, in fact, the invariants $\delta_R(K)$, $\overline{\delta}_R(K)$, $\underline{\delta}_R(K)$, and $\kappa_R(K)$ depend only on the real homology cobordism class of $S^3_0(K)$ by \cref{R-cob}.

\section{Excision theorem in real Seiberg--Witten theory}

In this section, we prove the satellite formula \cref{thm:mainhomotopytype}. To achieve this, we first establish an excision theorem in real Seiberg--Witten theory.

\subsection{Statement of the excision theorem and satellite formula}
Fix the convention of orientation so that the volume form is written as $dt \wedge d_Y$ on $[0,1]\times Y$.

\begin{thm}[Real Excision Theorem]\label{exthm}
    Let $Y$ and $Y'$ be $\mathbb{Z}_2$-homology 3-spheres with odd involutions. Suppose there exist $\mathbb{Z}_2$-equivariant decompositions:  
    \[
    Y = Y_1 \cup_{T^2} Y_2  \qquad\text{ and }\qquad Y' = Y'_1 \cup_{T^2} Y'_2,
    \]
    where the actions on $T^2$ are identified via $(z,w) \mapsto (-z, w)$. Define  
    \[
    Z = Y_1 \cup_{T^2} Y_2' \qquad \text{ and } \qquad Z' = Y_1' \cup_{T^2} Y_2.
    \]
    We suppose the restricted natural real spin structures on $T^2$ from $Y$, $Y'$, $Z$, and $Z'$ are the same up to sign.  
    Then, we have  
    \[
    SWF_R(Y) \wedge SWF_R(Y') \simeq SWF_R(Z) \wedge SWF_R(Z')
    \]
    as $G$-equivariant stable homotopy types with respect to natural real spin structures on $Z$ and $Z'$.
\end{thm}


As a corollary, we obtain \cref{thm:mainhomotopytype}. We recall the statement:

\begin{thm}\label{ex}
    Let $K$ be a knot in $S^3$. If $P$ is a pattern with an odd winding number, then the real Floer homotopy types
    \[
    SWF_R(P(K)) \qquad \text{and} \qquad SWF_R(K) \wedge SWF_R(P(U))
    \]
    are $G$-equivariantly stably homotopy equivalent. 
\end{thm}

\begin{proof}
    Let $Y$ and $Y'$ be the double branched covers of $P(K)$ and the unknot $U$, respectively. Then, by the assumption that $P$ has an odd winding number, we have the following $\mathbb{Z}_2$-equivariant decompositions:
    \begin{align*}
        Y &= Y_1 \cup_{T^2} Y_2 = \left(\widetilde{S^3 \setminus \nu(K)}\right) \cup_{T^2} \Sigma_2(S^1 \times D^2, P), \\ 
        Y' &= Y'_1 \cup_{T^2} Y'_2 = \left(\widetilde{S^3 \setminus \nu(U)}\right) \cup_{T^2} \Sigma_2(S^1 \times D^2, S^1 \times \{0\}),
    \end{align*} 
    where $\widetilde{S^3 \setminus \nu(K)}$ and $\widetilde{S^3 \setminus \nu(U)}$ are the 2-fold covers of the complements of open tubular neighborhoods of $K$ and $U$, respectively, and $\Sigma_2(S^1 \times D^2, P)$ and $\Sigma_2(S^1 \times D^2, S^1 \times \{0\})$ are the double branched covers of $S^1 \times D^2$ branched along $P$ and the trivial knot $S^1 \times \{0\}$, respectively. Note that the actions on $T^2$ are given by $(z, w) \mapsto (-z, w)$.

    Now, we apply \Cref{exthm}. Observe that we have 
    \[
    Z = Y_1 \cup_{T^2} Y_2' \qquad \text{ and } \qquad Z' = Y_1' \cup_{T^2} Y_2,
    \]
    so that $Z$ is the double branched cover of $K$ and $Z'$ is the double branched cover of $P(U)$. The conclusion then follows.
\end{proof}


We also prove \cref{degree} as a corollary. The statement is as follows:

\begin{cor}\label{degeqcor}
    For any knot $K$ in $S^3$ and any pattern $P$ with an odd winding number, we have  
    \[
    |\deg (K)|\cdot |\deg P(U)| = |\deg (P(K))| .
    \]
\end{cor}

\begin{proof}
    From \cref{ex}, we have 
    \[
    \widetilde{H}^*(SWF_R(P(K))) \cong \widetilde{H}^*(SWF_R(K)) \otimes \widetilde{H}^*(SWF_R(P(U))).
    \]
    Since it was shown in \cite[Proposition 5.2]{KPT:2024} that 
    \[
    |\deg (K)| = |\chi (\widetilde{H}^*(SWF_R(K)))|,
    \]
    the desired result follows.
\end{proof}

\subsection{Proof of the excision theorem}

\begin{figure}[htb]
    \centering \includegraphics[width=0.23\linewidth]{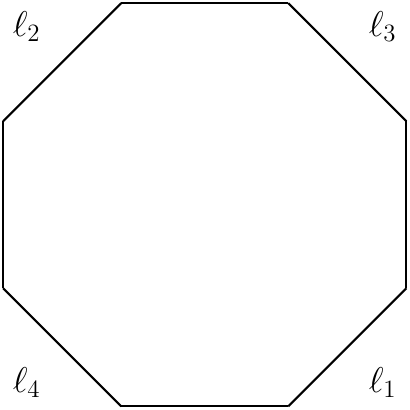}
    \caption{The manifold $U$ with eight codimension-one faces.}
    \label{fig:excision1}
\end{figure}
We use the same notation as in \Cref{exthm} and define a $\mathbb{Z}_2$-equivariant cobordism  
\[
W \colon Y \sqcup Y' \to Z \sqcup Z'
\]
as follows.  

First, we consider a 2-manifold $U$ with eight corners, which is diffeomorphic to $[0,1]^2$ and illustrated in \cref{fig:excision1}.
The manifold $U$ has eight codimension-one faces, four of which are labeled
$\ell_1, \ell_2, \ell_3$, and $\ell_4$. Next, we define the space  
\[
X^0 := \left( U \times T^2 \right) \cup_{\ell_1} 
\left( Y_1 \times [0,1] \right)
\cup_{\ell_2}
\left( Y_1' \times [0,1] \right)
\cup_{\ell_3}
\left( -Y_2 \times [0,1] \right)
\cup_{\ell_4}
\left( -Y_2' \times [0,1] \right),
\]
as described in \cref{fig:excision2}, where $T^2 = S^1 \times S^1$.   Then, we may regard $X^0$ as a cobordism from $Y \sqcup Y'$ to $Z \sqcup Z'$.  
We define an involution on $X^0$ by acting on the $T^2$-component via $(z,w) \mapsto (-z , w)$ and extending it naturally to each $Y_i$ and $Y_i'$.  
Let $\tau$ denote this $\mathbb{Z}_2$-action.

\begin{figure}[htb]
    \centering \includegraphics[width=0.35\linewidth]{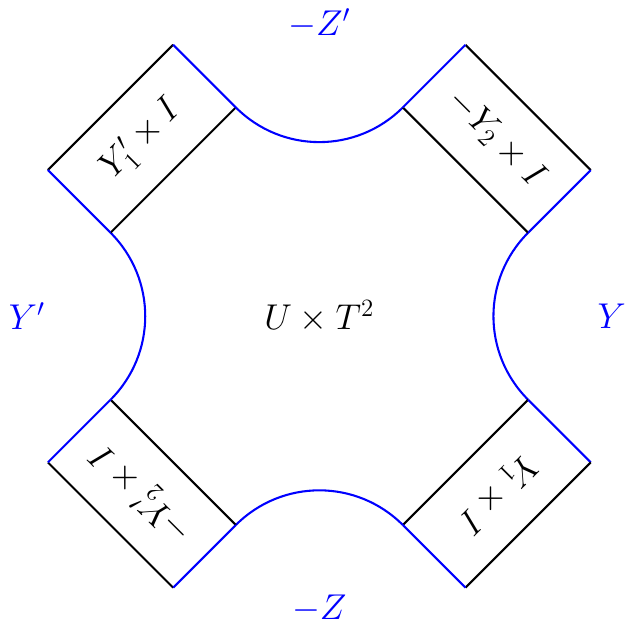}
    \caption{The manifold obtained by gluing $U \times T^2$ with $Y_1 \times [0,1]$, $Y_1' \times [0,1]$, $-Y_2 \times [0,1]$, and $-Y_2'\times [0,1]$.}
    \label{fig:excision2}
\end{figure}

\begin{lem}
    There is a unique real spin structure on $X^0$  up to isomorphism and sign. Furthermore, we have
    \[
    b_2 (X^0) - b_2 (X^0/ \tau) = 0, \qquad 
    b_1 (X^0) - b_1 (X^0/ \tau) = 0, \qquad 
    \text{ and } \qquad \sigma (X^0) = 0.
    \]
\end{lem}

\begin{proof}
    Note that if a manifold $M$ is equipped with a smooth $\mathbb{Z}_2$-action $\tau$, then we have  
    \[
    b_k(M) - b_k(M/\tau) = \dim_{\mathbb{Q}}(H_k(M; \mathbb{Q})^{-\tau^*}).
    \]
    Note that we have  
\[
H_i(T^2; \mathbb{Q})^{-\tau_*} = 0 \qquad \text{ and } \qquad
H_i(Y; \mathbb{Q})^{-\tau_*} = H_i(Y_1; \mathbb{Q})^{-\tau_*} \oplus H_i(Y_2; \mathbb{Q})^{-\tau_*} = 0,
\]
which follow from the Mayer--Vietoris exact sequence. Therefore we have 
    \[
    H_i(Y_j \times [0,1]; \mathbb{Q})^{-\tau^*} = H_i(Y_j' \times [0,1]; \mathbb{Q})^{-\tau^*} = H_i(U \times T^2; \mathbb{Q})^{-\tau^*} = 0
    \]
    and 
    \[
    H_i([0,1] \times T^2; \mathbb{Q})^{-\tau^*} = 0, 
    \]
    which implies that  $H_i(X^0; \mathbb{Q})^{-\tau^*} = 0$. 
    The calculation of the signature follows from a similar Mayer-Vietoris exact sequence.

Next, we discuss real spin structures.  
First, observe that if $T^2$ is equivariantly embedded in a $\mathbb{Z}_2$-homology 3-sphere $Y$ equipped with a real spin structure, then the induced real spin structure on $T^2$ is unique up to isomorphism. Divide $Y$ into two pieces, $Y_1$ and $Y_2$, along the embedded $T^2$.  
We define a real spin structure on $T^2$ as one that extends to a real spin structure on $S^1 \times T^2$, invariant under the $S^1$-rotation action.  
From the classification result in \cref{classification:real_spiny}, we know that there are two such real spin structures on $T^2$ up to sign.  
However, one of them does not extend to $Y_-$.  
Therefore, the real spin structures on $T^2$ induced from $Y$, $Y'$, $Z$, and $Z'$ are all the same.

Using this real spin structure on $U \times T^2$, we can induce a real spin structure on $X^0$.  
By \cref{classification:real_spin}, the uniqueness follows from the vanishing of $H^1(X^0/\tau; \mathbb{Z}_2)$, which can be seen from the Mayer--Vietoris exact sequence in $\mathbb{Z}_2$-cohomology.
\end{proof}

Then, using \cite{KMT:2024, Mi23}, we obtain the real Bauer--Furuta invariant
\[
BF^R_{X^0} \colon SWF_R(Y) \wedge SWF_R(Y') \to SWF_R(Z) \wedge SWF_R(Z')
\]
for the unique real spin structure on $X^0$.  
With the opposite orientation, we have
\[
BF^R_{-X^0} \colon SWF_R(Z) \wedge SWF_R(Z') \to SWF_R(Y) \wedge SWF_R(Y').
\]
The proof of the excision theorem \Cref{exthm} reduces to the following theorem. We remark that the proof follows essentially the same strategy as the excision cobordism argument by Kronheimer–Mrowka \cite{kronheimer2010knots}. Originally the excision theorem is stated by Floer in \cite{floer1990instanton} with a different proof.  

\begin{thm}\label{homotopyinv}
    The two maps $BF^R_{X^0}$ and $BF^R_{-X^0}$ are homotopy inverses of each other up to $G$-equivariant stable homotopy.
\end{thm}

\begin{proof}
    We consider the compositions $X^0 \circ -X^0$ and $-X^0 \circ X^0$. One can also glue the real spin structures on $X^0$ and $-X^0$ to obtain real spin structures on $X^0 \circ -X^0$ and $-X^0 \circ X^0$.  
    
    From \cref{gluing_boundary_component_2} below, we have 
    \[
        BF_{X_0}^R \circ BF_{-X_0}^R \cong BF_{X_0 \circ -X_0}^R \qquad \text{ and } \qquad  BF_{-X_0}^R \circ BF_{X_0}^R \cong BF_{-X_0 \circ X_0}^R. 
    \]
    Therefore we only have to prove that $BF_{X_0 \circ -X_0}^R$ and $BF_{-X_0 \circ X_0}^R$ are homotopic to the identity. Since these cases are symmetric, we focus only on $-X^0 \circ X^0$.  We perform fiberwise surgeries (whose fibers are $T^2$) along $S^1 \times T^2$, as described in \cref{fig:excision3}, and obtain  
    \[
    \left( (X^0 \circ -X^0 ) \setminus D^1 \times S^1 \times T^2 \right) \cup (S^0 \times D^2 \times T^2) \cong Z \times I \sqcup Z' \times I,
    \]
    where the $\mathbb{Z}_2$-action on $S^0 \times D^2 \times T^2$ is given by $\operatorname{id} \times \operatorname{id} \times \tau$ with $\tau (u,v) = (u, -v)$.  The real spin structure on $X^0 \circ -X^0$ extends naturally to the surgery  
    \[
    \left( (X^0 \circ -X^0 ) \setminus D^1 \times S^1 \times T^2 \right) \cup (S^0 \times D^2 \times T^2),
    \]
    which coincides with the product real spin structure on $Z \times I \sqcup Z' \times I$.  


Now, using the gluing formula \Cref{gluingformula} for the real Bauer--Furuta invariants from \cite{Mi23}, which is stated below, we obtain:
\begin{align*}
    BF^R_{X^0 \circ - X^0} &= BF^R_{X^0 \circ - X^0\setminus D^1\times S^1  \times T^2}  \circ_{SWF_R(T^3)} BF^R_{ D^1\times S^1  \times T^2} \\
    &= BF^R_{X^0 \circ - X^0\setminus D^1\times S^1  \times T^2} \\
    &= BF^R_{X^0 \circ - X^0\setminus D^1\times S^1  \times T^2}  \circ_{SWF_R(T^3)}  BF^R_{S^0 \times D^2 \times T^2}\\
    &= BF^R_{(X^0 \circ - X^0\setminus D^1\times S^1  \times T^2) \cup S^0 \times D^2 \times T^2} \\
    &= BF^R_{Z\times I \sqcup Z' \times I } = \id, 
\end{align*}
where:  
\begin{itemize}
    \item[(i)] The $G$-stable homotopy type $SWF_R(T^3)$ is the real Floer homotopy type for $T^3$ equipped with the involution  
    \[
    S^1\times S^1\times S^1\to S^1\times S^1\times S^1; \qquad (x,y,z) \mapsto (x, y, -z)
    \]
    with the real spin structure. Note that we have  
    \[
    b_1(T^3) -  b_1^\tau(T^3)=3-3=0.
    \]
    Thus, the real Floer homotopy type for $T^3$ with this involution can be defined analogously to the usual Floer homotopy type for a rational homology 3-sphere.
    
    \item[(ii)] The $G$-stable map  
    \[
    BF^R_{ D^1\times S^1  \times T^2} \colon SWF_R(T^3)  \to S^0
    \]
    is the real Bauer--Furuta invariant for $D^1\times S^1  \times T^2$ equipped with the involution  
    \[
    D^1\times S^1  \times T^2 \to D^1\times S^1  \times T^2; \qquad (t,x, y,z) \mapsto (t, x, y, -z).
    \]
    This 4-manifold with involution has a unique extension of the real spin structure of $T^3$ defined in (i).
    
    \item[(iii)] The $G$-equivariant stable map  
    \[
    BF^R_{X^0 \circ - X^0\setminus D^1\times S^1  \times T^2} \colon SWF_R(Y) \wedge SWF_R(Y') \to SWF_R(Z) \wedge SWF_R(Z') \wedge SWF_R(T^3)
    \]
    is the real Bauer--Furuta invariant for the restricted real spin structure on $X^0 \circ - X^0\setminus D^1\times S^1  \times T^2$ inherited from the real spin structure on $X^0 \circ - X^0$.

    \item[(iv)] The $G$-equivariant stable map  
    \[
    BF^R_{S^0 \times D^2 \times T^2} \colon SWF_R(T^3) \to S^0
    \]
    is the real Bauer--Furuta invariant for $S^0 \times D^2 \times T^2$ equipped with the involution  
    \[
    S^0 \times D^2 \times T^2 \to 
    S^0 \times D^2 \times T^2; \qquad (a, z,x,y) \mapsto (a, z, x,-y)
    \]
    and the real spin structure extending the real spin structure described in (i).

    \item[(v)] The $G$-equivariant stable map  
    \[
    BF^R_{Z\times I \sqcup Z' \times I } \colon SWF_R(Y) \wedge SWF_R(Y') \to SWF_R(Y) \wedge SWF_R(Y')
    \]
    is the real Bauer--Furuta invariant for the product real spin structure on $Z\times I \sqcup Z' \times I$.
\end{itemize}

From straightforward observations by taking flat or positive scalar curvature metrics with cylindrical metrics on the boundaries, we obtain  
\begin{align*}
    SWF_R(T^3) = S^0, \qquad 
    BF^R_{D^1\times S^1 \times T^2 } = \pm \operatorname{id} \colon S^0 \to S^0, \qquad \text{ and }\qquad 
    BF^R_{S^0 \times D^2 \times T^2} = \pm \operatorname{id} \colon S^0 \to S^0.
\end{align*}
However, the statement  
\[
BF^R_{Z\times I \sqcup Z' \times I } = \operatorname{Id}
\]
is nontrivial. This can be verified by an equivariant modification of the proof of $BF_{Y\times I} = \operatorname{id}$, which will appear in \cite{SSto}. This completes the proof.
\end{proof}

\begin{figure}[t]
    \centering \includegraphics[width=0.7\linewidth]{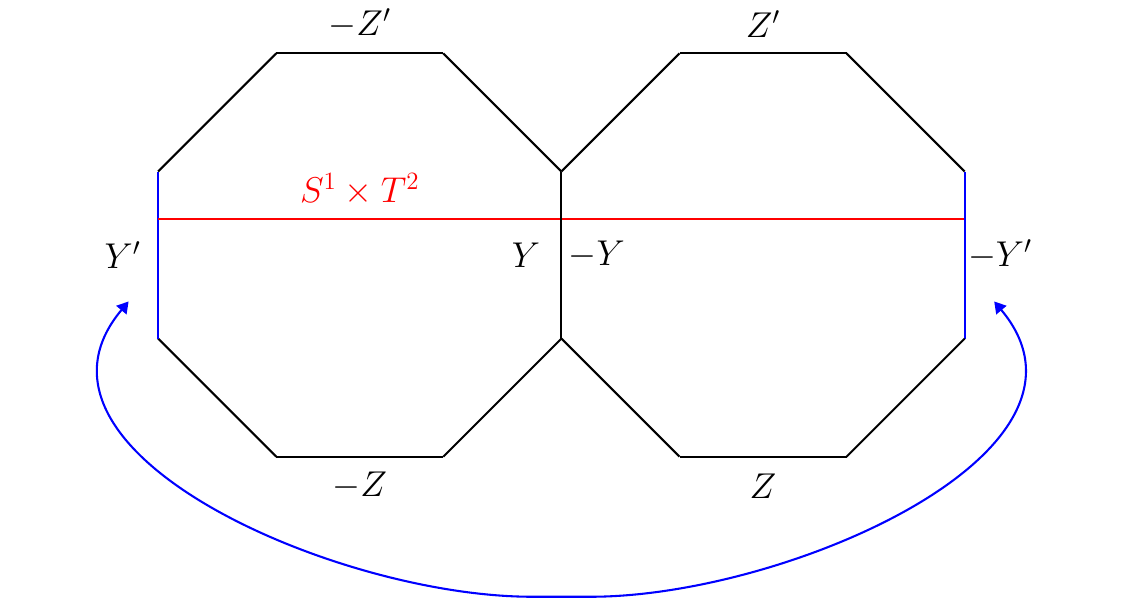}
    \caption{ The manifold $-X^0 \circ X^0$ and fiberwise surgeries.}
    \label{fig:excision3}
\end{figure}


We use the following type of gluing formula in real Seiberg--Witten theory, proven in \cite{Mi23}:  

\begin{thm}[{\cite[Theorem 2.12]{Mi23}}]\label{gen_glue}
Let $X_0$ and $X_1$ be compact, oriented $4$-manifolds with boundaries $\partial X_0 = Y$ and $\partial X_1 = -Y$.  
Suppose that $\tau_0$, $\tau_1$, and $\tau'$ are involutions on $X_0$, $X_1$, and $Y$, respectively, such that $\tau_0|_Y = \tau_1|_Y = \tau'$.  
Let $X := X_0 \cup_Y X_1$ and let $\tau := \tau_0 \cup_{\tau'} \tau_1$ be the induced involution on $X$.

Assume that
\[
H^1(X; \mathbb{R})^{-\tau^*} \cong H^1(X_0; \mathbb{R})^{-\tau_0^*} \cong H^1(X_1; \mathbb{R})^{-\tau_1^*} \cong H^1(Y; \mathbb{R})^{-{\tau'}^*} = 0,
\]
\[
H^0(X; \mathbb{Z})^{-\tau^*} = 0, \qquad \text{and} \qquad H^0(X_0; \mathbb{Z})^{-\tau_0^*} = 0.
\]
Also suppose that $X/\tau$, $X_0/\tau_0$, $X_1/\tau_1$, and $Y/\tau'$ are all connected.

Then the following statements hold:
\begin{itemize}
    \item Let $(\mathfrak{s}, I)$ be a real $\mathrm{spin}^c$ structure on $X$ that covers $\tau$, and let $(\mathfrak{t}, I)$ be the induced real $\mathrm{spin}^c$ structure on $Y$ via restriction. Then there exist unique relative real $\mathrm{spin}^c$ structures $(\mathfrak{s}_i, I_i)$ on $(X_i, Y)$ covering $\tau_i$ for $i = 0, 1$, such that 
    \[
    (\mathfrak{s}, I)|_{X_i} = (\mathfrak{s}_i, I_i), \qquad \text{and} \qquad (\mathfrak{s}_i, I_i)|_Y = (\mathfrak{t}, I).
    \]
    
    \item Let $\epsilon'$ denote the Spanier--Whitehead duality map in $\mathfrak{C}_G$. Then there exists a map
    \[
    \epsilon \colon 
    \pi_{V_{X_0}}^{\mathrm{st}}\left(\Sigma^{H^+(X_0)^{-\tau^*}} SWF(Y, \mathfrak{t}, I)\right) 
    \times 
    \pi_{V_{X_1}}^{\mathrm{st}}\left(\Sigma^{H^+(X_1)^{-\tau^*}} SWF(-Y, \mathfrak{t}, I)\right)
    \to 
    \pi_{V_X}^{\mathrm{st}}\left(S^{H^+(X)^{-\tau^*}}\right),
    \]
    given by composing the smash product with the map $\epsilon'$, where $V_{X_0}$, $V_{X_1}$, and $V_X$ are objects in $\mathfrak{C}_G$ defined by
    \[
    V_{X_i}^+ := \left[\left(\ker D_{\widetilde{g}}^+,\, 0,\, \dim_{\mathbb{C}}(\operatorname{coker} D_{\widetilde{g}}^+) + \frac{n(Y, \mathfrak{t}, g)}{2} \right)\right],
    \]
    in terms of the indices of AHS operators $D_{\widetilde{g}}$ or Dirac operators $D_{\widetilde{g}}^+$ with APS boundary conditions. Here, $\pi_V^{\mathrm{st}}$ denotes the equivariant stable homotopy group with respect to the representation $V$, and $n(Y, \mathfrak{t}, g)$ is the correction term defined in~\cite{Mi23}. Moreover, the real Bauer--Furuta invariant satisfies the following gluing formula:
    \[
    BF^R(X, \mathfrak{s}, I) = \epsilon\left(BF^R(X_0, \mathfrak{s}_0, I_0) \wedge BF^R(X_1, \mathfrak{s}_1, I_1)\right).
    \]
\end{itemize}
\end{thm}

\begin{cor}\label{gluingformula}
Let $X_0$ and $X_1$ be real spin $4$-manifolds satisfying the assumptions of \cref{gen_glue}, each having a boundary component identified with $T^3$, such that the restrictions of their real spin structures to $T^3$ coincide with the one described in \textup{(i)} of the proof of \cref{homotopyinv}. Let $X := X_0 \cup_{T^3} X_1$ be equipped with the induced $\mathbb{Z}_2$-action and the glued real spin structure. Then
\[
BF^R_X = BF^R_{X_0} \circ_{SWF_R(T^3) = S^0} BF^R_{X_1}
\]
as $G$-equivariant stable maps, up to $G$-equivariant homotopy.

An analogous statement holds when $T^3$ is replaced with $S^1 \times S^2$, equipped with the involution
\[
S^1 \times S^2 \to S^1 \times S^2, \qquad (x, y) \mapsto (-x, y),
\]
and the unique real spin structure.
\end{cor}

To prove \cref{homotopyinv}, we use the following version of the gluing theorem, which is not covered by \cref{gluingformula}.

\begin{lem}\label{gluing_boundary_component_2}
Let $X_0$, $X_1$, and $X$ be $4$-manifolds, and let $\tau_0$, $\tau_1$, and $\tau$ be involutions on $X_0$, $X_1$, and $X$, respectively.  Suppose that $\partial X_0 = -\partial X_1 = Y$ is a $3$-manifold with two connected components $Y = Y_0 \cup Y_1$, and let $\tau'$ be an involution on $Y$ that preserves each component. Assume that $X_0$, $X_1$, $X$, $Y$, $\tau_0$, $\tau_1$, $\tau$, and $\tau'$ satisfy the hypotheses of \cref{gluingformula}, except for the condition that $Y/\tau'$ is connected. In particular, assume that $H^0(Y; \mathbb{Z})^{-{\tau'}^*} = 0$. Then the real Bauer--Furuta invariant satisfies the gluing formula:
\[
BF^R(X, \mathfrak{s}, I) = \epsilon\left(BF^R(X_0, \mathfrak{s}_0, I_0) \wedge BF^R(X_1, \mathfrak{s}_1, I_1)\right).
\]
\end{lem}

\begin{proof}
The proof of this lemma is identical to that of \cref{gluingformula}, except for Step 5. In Step 5, we consider the homotopy of the boundary condition operator $D_H$. Here, we must carefully distinguish between the two spaces
\[
i\Omega^0(Y)^{-{\tau'}^*} \qquad \text{and} \qquad i\Omega^0_0(Y)^{-{\tau'}^*} := \left\{ f \in i\Omega^0(Y)^{-{\tau'}^*} \,\middle|\, \int_{Y_j} f = 0 \text{ for all connected components } Y_j \subset Y \right\}
\]
in order to construct a $1$-parameter family of gauge transformations to deform the boundary condition.

If $Y$ is connected, then these two spaces coincide, and Step 5 proceeds without difficulty. This is the only step where the assumption that $Y/\tau'$ and $X_1/\tau_1$ are connected is used. However, if we assume that $H^0(Y; \mathbb{Z})^{-{\tau'}^*} = 0$, then we again have
\[
i\Omega^0(Y)^{-{\tau'}^*} = i\Omega^0_0(Y)^{-{\tau'}^*},
\]
so the gluing formula can be proved in exactly the same manner as in the proof of \cref{gluingformula}.
\end{proof}


\section{Proof of the applications}\label{Proof of the applications}

\subsection{Real $10/8$-type inequalities}\label{real 10/8 inequalities}
We first present a slightly stronger version of the real $10/8$-type inequality compared to \cite[Theorem 1.3]{KMT:2021}. Before stating the result, we introduce the following definition:

\begin{defn}
    We say that the real Floer homotopy type of $SWF_R(K)$ is \emph{strongly spherical} if 
    $
     SWF_R(K) $
    is $G$-equivariant stable homotopy equivalent to $\mathbb{C}_+^m$ for some $m$.
\end{defn}

Now, we state the inequality. 
\begin{thm}\label{thm:KGsplit}
Let $K$ and $K'$ be knots in $S^3$, and let $X$ be an oriented smooth compact connected cobordism from $S^3$ to $S^3$ with $H_1(X) = 0$. Suppose that $S$ is an oriented smooth compact connected properly embedded cobordism in $X$ from $K$ to $K'$ such that the homology class $[S]$ of $S$ is divisible by $2$ and that  
the homology class $[S]/2$ in $H_2(X, \partial X)$ reduces to
$w_2(X)$. Then, we have  
\[
\kappa_R(K) - \frac{1}{16} \sigma(\Sigma_2(S)) \leq \kappa_R(K') + b^+(\Sigma_2(S)) - b^+(X).
\]
Moreover, if $b^+(\Sigma_2(S)) - b^+(X) \geq 1$ and the real Floer homotopy type of $SWF_R (K)$ is  strongly spherical,
then we have  
\[
\kappa_R(K) - \frac{1}{16} \sigma(\Sigma_2(S)) + \frac{1}{2} \leq \kappa_R(K') + b^+(\Sigma_2(S)) - b^+(X).
\]
\end{thm}
The first inequality in \cref{thm:KGsplit} is already follows from \cite[Theorem 1.3]{KMT:2021}. In this section, we shall prove the second inequality. 

\begin{rem}
For the original relative $10/8$-type inequality, Manolescu \cite{manolescu2014intersection} established a $+1$-version under the $K_{{\rm Pin(2)}}$-splitting condition. In this setting, one has
\[
\widetilde{K}_{{\rm Pin(2)}}(S^0) \cong \mathbb{Z}[w, z]/(w^2 - 2w,\, wz = 2w).
\]
The $K_{{\rm Pin(2)}}$-splitting condition is defined as follows: an arbitrary element of the ideal
\[
\operatorname{Im} \left(\widetilde{K}_{{\rm Pin(2)}}(SWF(Y)) \to \widetilde{K}_{{\rm Pin(2)}}(SWF(Y)^{S^1}) \cong \mathbb{Z}[w, z]/(w^2 - 2w,\, wz = 2w) \right)
\]
must be of the form
\[
z^k(\lambda w + P(z)) = \lambda 2^k w + z^k P(z),
\]
where $\lambda \in \mathbb{Z}$ and $P(z)$ is a polynomial in $z$.

In the setting of the real $10/8$-type inequality, we replace ${\rm Pin}(2)$ with $G = \langle j \rangle$.  
Accordingly, the representation ring $R({\rm Pin}(2))$ is replaced by
\[
R(G) = \mathbb{Z}[w, z]/(w^2 - 2w,\, w - 2z + z^2).
\]
When $K$ and $K'$ are the unknots, the corresponding statement in the real setting is proved in \cite{konno2022dehn}.
\end{rem}

Before proceeding to the proof, we briefly recall the setting of the real $10/8$-type inequality.  
A pointed finite $G$-CW complex $X$ is called a \emph{space of type $G$-SWF} if:
\begin{itemize}
    \item The fixed-point set $X^H$ is $G$-homotopy equivalent to $(\tilde{\mathbb{R}}^s)^+$ for some $s \geq 0$.
    \item The $G$-action on $X \setminus X^H$ is free.
\end{itemize}
The natural number $s$ is called the \emph{level} of $X$. For a space of type $G$-SWF, define
\[
k(X) := \min \left\{ k \geq 0 \mid \exists x \in \mathfrak{J}(X) \text{ such that } wx = 2^k w \right\},
\]
where $\mathfrak{J}(X)$ is an ideal of $R(G)$ defined by the property that the image of 
\[
\tau^\ast \colon \widetilde{K}_G(X) \to \widetilde{K}_G(X^H)
\]
is equal to $\mathfrak{J}(X) \cdot b_{t\tilde{\mathbb{C}}}$, where $b_{t\tilde{\mathbb{C}}}$ denotes the Bott element (see \cite[Section 2.1]{manolescu2014intersection}).

The following doubling construction is used to define the real $\kappa$-invariant. Define a group automorphism $\alpha \colon G \to G$ by $\alpha(j) = -j$.
\begin{defn}\label{double}
Let $X$ be a space of type $G$-SWF at level $t$.  
Define $X^\dagger$ to be the space of type $G$-SWF at level $t$ given by the same underlying topological space as $X$, but with the $G$-action twisted by $\alpha$, that is, the action is given by composing the original $G$-action with $\alpha$. Then the smash product $X \wedge X^\dagger$ is also a space of type $G$-SWF, now at level $2t$.  
We define the \emph{double} of $X$ to be the space of type $G$-SWF:
\[
D(X) := X \wedge X^\dagger.
\]

Similarly, for a real or complex representation $V$ of $G$, define the twisted representation $V^\dagger$ to be the same underlying vector space as $V$, but with $G$-action given by composing the original action with $\alpha$.  
Define the \emph{double} of $V$ by
\[
D(V) := V \oplus V^\dagger.
\]
\end{defn}

In order to prove the latter statement in \cref{thm:KGsplit}, we establish the following:

\begin{prop}\label{+1keylem}
Let $X$ and $X'$ be spaces of type $G$-SWF at levels $t$ and $t'$, respectively, with $t < t'$.  
Suppose $X$ is stably homotopy equivalent to $(\C_+^m)^+ \wedge (\R^t)^+ $ and there exists a pointed $G$-equivariant map $f \colon X \to X'$ whose restriction to the $G$-fixed-point sets is a homotopy equivalence.  
Then we have
\[
k(DX) + t + \frac{1}{2} \leq k(DX') + t'.
\]
\end{prop}

\begin{proof}
The map $f$ induces a map 
\[
Df \colon DX \to DX'. 
\]
We begin with the following commutative diagram:
\begin{align}
\label{diagram k}
  \begin{CD}
     \widetilde{K}_G(DX')   @>{Df^\ast}>> \widetilde{K}_G(DX) \\
  @V{(Di')^*}VV    @V{Di^*}VV \\
\widetilde{K}_G((DX')^H)   @>{(Df^{H})^\ast}>> \widetilde{K}_G(DX^H) 
  \end{CD}
\end{align}
Since $\widetilde{K}_G((DX')^H)$ and $\widetilde{K}_G(DX^H)$ are free $R(G)$-modules of rank $1$, we may view the maps in the diagram as multiplications by elements of $R(G)$. The vertical maps are given by multiplication with the $K$-theoretic Euler classes $w^{t'}$ and $w^{t}$, respectively.  
It is shown in \cite{KMT:2021} that $(Df^H)^\ast$ in diagram \eqref{diagram k} corresponds to multiplication by $2^{t'-t-1}w$ when we suppose $t'-t>0$.

Now take $x \in \mathfrak{J}(DX')$ such that $w x = 2^{k(X')} w$.  
By the commutativity of \eqref{diagram k}, we have
\begin{align*}
2^{t'-t-1 + k(X')}w = 2^{t'-t-1}w \cdot x = (Df^H)^* x = i^* Df^* (y),
\end{align*}
where $y$ is an element satisfying $(i')^* y = x$.   On the other hand, we are assuming 
\[
DX = (\tilde{\mathbb{C}}^t)^+ \wedge (\mathbb{C}_+^s \oplus \mathbb{C}_-^s)^+
\]
for some $s \geq 0$. In this case, we have $k(X) = s$.  
Under this identification, the map $i^*$ corresponds to multiplication by $(z\overline{z})^s$.  
Hence,
\begin{align}\label{oo}
2^{t'-t-1 + k(X')}w = (z\overline{z})^s \cdot Df^*(y) = (z\overline{z})^s \cdot h,
\end{align}
where we set $h := Df^*(y)$.

Since $t' - t > 0$, the left-hand side of \eqref{oo} lies in the kernel of the forgetting map
\[
R(G) = \mathbb{Z}[w, z]/(w^2 - 2w,\, w - 2z + z^2) \longrightarrow R(H) = \mathbb{Z}[z]/(-2z + z^2).
\]
Therefore, $(z\overline{z})^s h$ also lies in the kernel. In particular, $h$ has the form
\[
h(w, z) = g(z - 2) + w g'
\]
for some $g, g' \in \mathbb{Z}[w, z]/(w^2 - 2w,\, w - 2z + z^2)$. We now consider the doubling of \eqref{oo}:
\begin{align}\label{ooo}
2^{2t' - 2t + 2k(DX') - 1}w = (z\overline{z})^{2s} \cdot h \cdot \overline{h}.
\end{align}
For a polynomial $P \in \mathbb{Z}[w, z]/(w^2 - 2w,\, w - 2z + z^2)$, define
\[
\overline{P}(w, z) := P(w, \overline{z}),
\quad \text{where} \quad \overline{z} = z + w - w z.
\]
We compute:
\begin{align*}
h \cdot \overline{h} = w q + g \cdot \overline{g} (z - 2)(\overline{z} - 2) = w q + g \cdot \overline{g} (z \overline{z} + 2 q')
\end{align*}
for some polynomials $q$ and $q'$ in $\mathbb{Z}[w, z]/(w^2 - 2w,\, w - 2z + z^2)$. Plugging this into \eqref{ooo} and multiplying both sides by $w$, we get:
\begin{align*}
2^{2t' - 2t + 2k(DX')}w &= w \cdot (z\overline{z})^{2s} \cdot \left(w q + g \cdot \overline{g}(z \overline{z} + 2 q')\right) \\
&= 2^{2s} w \cdot \left(w q + g \cdot \overline{g} \cdot z \overline{z} + 2 q'\right) \\
&= 2^{2k(DX) + 1} \cdot \left(q + g \cdot \overline{g} + q'\right).
\end{align*}
Thus we obtain the inequality
\[
k(DX') + t' - t \geq k(DX) + \tfrac{1}{2}.
\]
This completes the proof.
\end{proof}

\begin{proof}[Proof of \cref{thm:KGsplit}]
The first inequality follows from \cite{KMT:2021}. We now consider the second inequality. Recall that the definition of the \emph{$K$-theoretic Fr{\o}yshov invariant} is given by
\[
\kappa_R(K) := k\left(\Sigma^{-D(V^0_\lambda)}
\Sigma^{-D(W^0_\lambda)}
\Sigma^{-D((W^\dagger)^0_\lambda)} D\left(I^\mu_\lambda(\Sigma_2(K), \mathfrak{t}, \iota, g)\right)\right) - \frac{n(Y, \mathfrak{t}, g)}{2},
\]
where the notations follow \cite{KMT:2021}.

Now, set $\Sigma_2(K) = Y_0$, $\Sigma_2(K') = Y_1$, and $\Sigma_2(S) = W$.  
From the assumptions, $W$ admits a real spin structure.  
Thus, we obtain a $G$-equivariant map of the form:
\[
f \colon \Sigma^{m_0 \tilde{\mathbb{R}}} \Sigma^{n_0^+ \mathbb{C}_+} \Sigma^{n_0^- \mathbb{C}_+} I_{-\mu}^{-\lambda}(Y_0) \to \Sigma^{m_1 \tilde{\mathbb{R}}} \Sigma^{n_1^+ \mathbb{C}_+} \Sigma^{n_1^- \mathbb{C}_+} I^\mu_\lambda(Y_1),
\]
where $I^\mu_\lambda(Y_i) := I^\mu_\lambda(Y_i, \mathfrak{t}_i, \iota_i, g_i)$, and $m_i, n_i^\pm \geq 0$, with $-\lambda, \mu$ taken sufficiently large. Taking the double of $f$, we obtain the \emph{doubled cobordism map}, or the \emph{doubled relative Bauer--Furuta invariant},
\begin{align}
\label{eq: doubled cob map}
D(f) \colon 
\Sigma^{m_0 \tilde{\mathbb{C}}} \Sigma^{n_0 (\mathbb{C}_+ \oplus \mathbb{C}_-)} D\left(I_{-\mu}^{-\lambda}(Y_0)\right) \to 
\Sigma^{m_1 \tilde{\mathbb{C}}} \Sigma^{n_1 (\mathbb{C}_+ \oplus \mathbb{C}_-)} D\left(I^\mu_\lambda(Y_1)\right),
\end{align}
where $n_i := n_i^+ + n_i^-$. Denote by $V_i(\tilde{\mathbb{R}})_\lambda^\mu$ the vector space $V(\tilde{\mathbb{R}})_\lambda^\mu$, which is a finite-dimensional approximation of the $I$-invariant part for $Y_i$.  We use similar notation for other representations as well.

We have
\begin{align*}
m_0-m_1=&\dim_\R (V_1(\tilde{\R})^0_{\lambda})-\dim_\R(V_0(\tilde{\R})^0_{-\mu}) - b^+(W) + b^+_\iota(W), \\
\begin{split}
n_0-n_1=&\dim_\C(V_1(\C_+)^0_{\lambda}) +\dim_\C(V_1(\C_-)^0_{\lambda})\\
&-\dim_\C(V_0(\C_+)^0_{-\mu})-\dim_\C(V_0(\C_-)^0_{-\mu})\\
&- \frac{\sigma(W)}{16}+ \frac{n(Y_1, \mathfrak{t}_1, g_1)}{2}-\frac{n(Y_0, \mathfrak{t}_0, g_0)}{2}.
\end{split}
\end{align*}
We apply \cref{+1keylem} to \eqref{eq: doubled cob map}. 
Note that $t'-t$ corresponds to $b^+(W)- b^+_\iota(W)$. This  completes the proof. \end{proof}

\subsection{Proof of main theorems}

We now proceed to prove \cref{thm:topologicalmain} and \cref{thm:topologicalmaingeneral}, which are implied by the following theorem:

\begin{thm}
Let $P$ be a pattern with odd winding number, and let $m$ be a positive integer.  
Then any finite self-connected sum of $P(mE_{2,1}) \mathbin{\#} -P(U)$ does not bound a normally immersed disk in $B^4$ with only negative double points.
\end{thm}

\begin{proof}
We first prove that for any positive integer $m$,  
\begin{equation}\label{eq:kappa21}
    \frac{1}{2} \leq \kappa_R( -m E_{2,1} ),
\end{equation}
by applying \Cref{thm:KGsplit}. There exists an oriented, compact, connected, properly, and smoothly embedded cobordism $S_m$ in the twice-punctured $2m\mathbb{CP}^2$, denoted by $X$, from $mT_{2, 19}$ to $-m E_{2,1}$~\cite{ACMPS:2023}. Moreover, $S_m$ represents the homology class
\[
(\underbrace{2, \dots, 2}_{m},\, \underbrace{6, \dots, 6}_{m}) \in H_2(X, \partial X) \cong \mathbb{Z}^{2m}.
\] The Mayer–Vietoris sequence and the $G$-signature theorem (see, e.g., \cite[Lemma 4.5]{KMT:2024}) yield
\[
\begin{split}
    b^+(\Sigma_2(S_m)) - b^+(X) 
    &= b^+(X) -\frac{1}{4} [S_m]^2 - \frac{1}{2} \sigma (mT_{2,19})\\
    &= 2m - \frac{1}{4} \left( (2)^2m + (6)^2m \right) + \frac{1}{2} (18m) \\
    &= m.
\end{split}
\]
Similarly, we compute  
\[
\begin{split}
    \sigma(\Sigma_2(S_m))
    &= 2\sigma(X) -\frac{1}{2} [S_m]^2 - \sigma(mT_{2,19})\\
    &= 4m - \frac{1}{2} \left( (2)^2m + (6)^2m \right) + 18m \\
    &= 2m.
\end{split}
\]
Additionally, by \cite[Theorem 1.6 and Theorem 1.8]{KMT:2021}, we have  
\[
\kappa_R(m T_{2, 19}) = \frac{9}{8} m. 
\]
Finally, since the real Floer homotopy type of \( mT_{2,19} \) is strongly spherical~\cite[Proposition 3.53]{KMT:2021}, we apply the latter inequality of \Cref{thm:KGsplit} to obtain  

\[
m+\frac{1}{2} =\frac{9}{8} m - \frac{1}{8}m + \frac{1}{2}  \leq \kappa_R(-m E_{2,1}) + m,
\]
which proves the claim.

Next, we prove that this implies that $m E_{2,1}$ does not bound a normally immersed disk in $B^4$ with only negative double points. Suppose, for contradiction, that it does, with $\ell$ negative double points. Then, there exists a smooth concordance $S'$ in the twice-punctured $\ell\mathbb{CP}^2$, denoted by $X'$, from $-m E_{2,1}$ to the unknot. Moreover, $S'$ represents the homology class  
\[
[S'] = (2,2, \ldots, 2) \in H_2(X', \partial X').
\]
Applying the first inequality of \Cref{thm:KGsplit} in a similar manner as above, we conclude that  
\[
\kappa_R(-m E_{2,1}) \leq 0,
\]
which is a contradiction.

Finally, \Cref{thm:mainhomotopytype} implies that the connected sum of any finite number of copies of $P(mE_{2,1}) \mathbin{\#} -P(U)$ lies in the same local equivalence class as the connected sum of the same number of copies of $mE_{2,1}$. We denote such a connected sum by $J$. In particular, by \Cref{cor:mainhomotopytype} together with \eqref{eq:kappa21}, we conclude that $-J$ has $\kappa_R$-invariant at least $1/2$. Hence, the same argument as above applies to $J$, completing the proof.
\end{proof}

\begin{rem}
Since we have constructed a geometric real spin cobordism
\[
X^0 \colon \Sigma_2(P(mE_{2,1})) \to \Sigma_2(mE_{2,1}) \sqcup \Sigma_2(P(U))
\]
satisfying
\begin{align*}
b^2(X^0) - b^2_\tau(X^0) &= 0, \\
\sigma(X^0) &= 0,
\end{align*}
we can slightly refine the non-sliceness result for $P(mE_{2,1}) \# -P(U)$, for any pattern $P$ with odd winding number, as follows:

If $m = 3$, then
\[
\mathrm{sn}(P(mE_{2,1}) \# -P(U)) \geq 2,
\]
and if $m = 6$, then
\[
\mathrm{sn}(P(mE_{2,1}) \# -P(U)) \geq 3.
\]
The proof is similar to the argument in \cite{FT25}, combined with the cobordism $X^0$.
\end{rem}


Finally, we prove \cref{thm:stabnumber}, whose statement we recall below.

\begin{thm}\label{thm:stabnumberbody} Suppose that $K$ is a finite self-connected sum of $T_{3,11}$. Then, for any pattern $P$ with an odd winding number such that $P(U)$ is the unknot, we have
\[
\lim_{n \to \infty} \left( \mathrm{sn}(n P(K)) - \mathrm{sn}^{\mathrm{Top}}(n P(K)) \right) = \infty.
\]
\end{thm} 

\begin{proof}
    Suppose that $K$ is $mT_{3,11}$ for some positive integer $m$. In \cite{BBL:2020}, Baader, Banfield, and Lewark proved that $mT_{3,11}$ can be transformed into the unknot using $8m$ null-homologous twists. Recall that a \emph{null-homologous twist} is an operation on an oriented knot that inserts a full twist into $2\ell$ parallel strands, where $\ell$ strands are oriented upwards and $\ell$ strands are oriented downwards. By the work of McCoy~\cite[Theorem 1.1]{McCoy:2021}, this implies that the \emph{algebraic genus} $g_\mathbb{Z}^{\mathrm{Top}}$, defined as the minimal genus of a locally flat oriented surface embedded in $B^4$ that is bounded by the knot and whose complement has fundamental group $\mathbb{Z}$, for $mT_{3,11}$ is at most $8m$.  

    Next, we apply the inequality
\[
g_\mathbb{Z}(P(K)) \leq g_\mathbb{Z}(P(U)) + g_\mathbb{Z}(K),
\]
which is due to McCoy~\cite[Theorem 1.4]{McCoy:2021} and Feller, Miller, and Caicedo~\cite[Theorem 1.1]{FMP:2022}. From this, we obtain
\[
g_4^{\mathrm{Top}}(P(K)) \leq g_\mathbb{Z}(P(K)) \leq 8m,
\]
for any pattern $P$ such that $P(U)$ is the unknot. Here, $g_4^{\mathrm{Top}}(K)$ denotes the standard \emph{topological slice genus} of a knot $K$, defined as the minimal genus of a locally flat oriented surface embedded in $B^4$ bounded by $K$.

 By a theorem of Conway and Nagel~\cite[Theorem 5.15]{CN:2020}, which states that for a knot $K$ with $\mathrm{Arf}(K) = 0$, the stabilizing number $\mathrm{sn}^{\mathrm{Top}}(n P(K))$ is bounded above by $g_4^{\mathrm{Top}} (nP(K))$, we conclude that 
    \[
    \mathrm{sn}^{\mathrm{Top}}(n P(K)) \leq g_4^{\mathrm{Top}} (nP(K)) \leq ng_4^{\mathrm{Top}} (P(K))\leq 8nm
    \]
    for each positive integer $n$ and $m$.

    Now, we will prove that  
\[
9nm \leq \mathrm{sn}(n P(K)),
\]  
which will complete the proof. To do so, we apply the first inequality of \cref{thm:KGsplit}, as in \cite[Theorem 1.10]{KMT:2021}.  
Let \( p = \mathrm{sn}(n P(K)) \), and take \( X \) to be the twice-punctured \( pS^2 \times S^2 \), and \( S \) to be a null-homologous disk bounded by \( n P(K) \) in \( pS^2 \times S^2 \). The inequality then yields:  
\begin{equation}\label{eq:sn}
    -\frac{9}{16} \sigma(n P(K)) - \kappa_R(n P(K)) \leq \mathrm{sn}(n P(K)).    
\end{equation}  
A formula of Litherland~\cite[Theorem 2]{Litherland:1979} implies that  
\[
\sigma(n P(K)) = \sigma(n K) = \sigma(nm T_{3,11}) = -16nm.
\]  
Finally, by \cref{thm:mainhomotopytype}, we have that the local equivalence class of \( n P(K) \) is equal to the local equivalence class of \( nm T_{3,11} \). Therefore, combining this with \cref{cor:mainhomotopytype} and \cite[Theorem 1.7 and Theorem 1.8]{KMT:2021}, we obtain  
\[
\kappa_R(n P(K)) = \kappa_R(nm T_{3,11}) = nm \left(\kappa_R( T_{3,11})\right)  
= nm \left(-\frac{1}{2}\overline{\mu}( \Sigma(2,3,11)) \right) = 0,
\]  
where \( \overline{\mu} \) is the Neumann-Siebenmann invariant. Applying the above computations to \eqref{eq:sn} completes the proof.
\end{proof}

\section{Invariants of homology $S^1\times S^3$}\label{homologys1s3}

Let $X$ be an oriented homology $S^1 \times S^3$, and let $\widetilde{X}$ denote its unique nontrivial double cover. A similar argument to the homology $S^1 \times S^2$ case yields the following:

\begin{lem}
There is a unique real spin$^c$ structure on $\widetilde{X}$ up to isomorphism and sign.
\end{lem}

\begin{proof}
This follows by an argument similar to that of \cref{s1s2}. Let $l$ be the nontrivial real line bundle over $X$ corresponding to the double cover $\widetilde{X} \to X$. Then we have
\[
w_2(X) + w_1(l)^2 = 0.
\]
By \cref{classification:real_spin}, this implies that $\widetilde{X}$ admits a real spin structure.

Moreover, by the same classification, the isomorphism classes of such structures correspond to real line bundles on $\widetilde{X}$ equipped with lifts of the involution. However, this construction lifts to a complex line bundle over $X$, and all such bundles are topologically trivial. Hence, the real spin$^c$ structure on $\widetilde{X}$ is unique.
\end{proof}

With respect to the real spin$^c$ structure on $\widetilde{X}$, we have the associated real Bauer--Furuta invariant  
\[
BF^{R}_{\widetilde{X}} \colon S^{-\frac{1}{16} \sigma(\widetilde{X})} = S^0 \to S^{b^+(\widetilde{X}) - b^+(X)} = S^0.
\]
We define the invariant
\[
|\deg(X)| \in \mathbb{Z}_{\geq 0}
\]
to be the degree of the real Bauer--Furuta invariant of $X$ with respect to the real spin$^c$ structure. One can easily verify the following.

\begin{lem}\label{PSC}
The value $|\deg(X)|$ is invariant under orientation preserving diffeomorphisms between oriented homology $S^1 \times S^3$ manifolds.  
Moreover, if $X$ admits a positive scalar curvature metric, then
\[
|\deg(X)| = 1.
\]
\end{lem}

\begin{proof}
The invariance under diffeomorphisms is routine.

Suppose now that $X$ admits a positive scalar curvature metric.  
Using the standard argument involving the Weitzenböck formula, we see that there are no real solutions to the Seiberg--Witten equations on $\widetilde{X}$ for any real spin$^c$ structure. This implies that only the reducible solution exists. In particular, the real Bauer--Furuta invariant $BF^R_{\widetilde{X}}$ has mapping degree $\pm 1$, so $|\deg(X)| = 1$. This completes the proof.
\end{proof}

Recall that for a given 2-knot $K$ in $S^4$, we denote by $X(K)$ the 4-manifold obtained by performing surgery along $K$.  
We now examine the relationship between $|\deg(K)|$ and $|\deg(X(K))|$.  
Although we assume that the 2-knot lies in $S^4$, the proof in fact applies to any 2-knot in a homology $S^4$.

\begin{prop}\label{degeq}
For any smooth $2$-knot $K$ in $S^4$, we have
\[
|\deg(K)| = |\deg(X(K))|.
\]
\end{prop}

\begin{proof}
This follows from the standard gluing argument in real Seiberg--Witten theory; see \cite{Mi23}. We briefly sketch the proof.

Decompose $S^4$ as
\[
S^4 = \left( S^4 \setminus \nu(K) \right) \cup_{S^1 \times S^2} \nu(K),
\]
and similarly decompose the double branched cover as
\[
\Sigma_2(K) = \left( \widetilde{S^4 \setminus \nu(K)} \right) \cup_{S^1 \times S^2} \Sigma_2(S^2 \times D^2, S^2 \times \{0\}).
\]
Here, $\nu (K)$ is a closed neighborhood of $K $ and the involution on $S^1 \times S^2$ is given by $(x, y) \mapsto (-x, y)$.

On the other hand, the surgery $X(K)$ admits the decomposition
\[
X(K) = \left( S^4 \setminus \nu(K) \right) \cup_{S^1 \times S^2} D^3 \times S^1.
\]
Taking the double cover, we obtain
\[
\widetilde{X(K)} = \left( \widetilde{S^4 \setminus \nu(K)} \right) \cup_{S^1 \times S^2} \widetilde{D^3 \times S^1}.
\]
Note that the gluing region $S^1 \times S^2$ is equipped with the $\mathbb{Z}_2$-action $\tau \colon (x, y) \mapsto (-x, y)$. Under this involution, one can verify that the real Floer homotopy type of $S^1 \times S^2$ is $S^0$.

With respect to these decompositions, we have
\begin{align*}
BF^R_{\Sigma_2(K)} &= BF^R_{\widetilde{S^4 \setminus \nu(K)}} \wedge_{SWF_R(S^1 \times S^2)} BF^R_{\Sigma_2(S^2 \times D^2, S^2 \times \{0\})}, \\
BF^R_{X(K)} &= BF^R_{\widetilde{S^4 \setminus \nu(K)}} \wedge_{SWF_R(S^1 \times S^2)} BF^R_{\widetilde{D^3 \times S^1}},
\end{align*}
where:
\begin{itemize}
    \item The notation $SWF_R(S^1 \times S^2)$ denotes the real Floer homotopy type of $S^1 \times S^2$ equipped with the involution
    \[
    S^1 \times S^2 \to S^1 \times S^2; \qquad (x, y) \mapsto (-x, y).
    \]
    
    \item The stable map
    \[
    BF^R_{\widetilde{S^4 \setminus \nu(K)}} \colon S^0 \to SWF_R(S^1 \times S^2)
    \]
    is the real Bauer--Furuta invariant of $\widetilde{S^4 \setminus \nu(K)}$, defined using the restricted real spin structure.
    
    \item The stable map
    \[
    BF^R_{\Sigma_2(S^2 \times D^2, S^2 \times \{0\})} \colon SWF_R(S^1 \times S^2) \to S^0
    \]
    is the real Bauer--Furuta invariant of $\Sigma_2(S^2 \times D^2, S^2 \times \{0\})$, defined using the restricted real spin structure.
    
    \item The stable map
    \[
    BF^R_{\widetilde{D^3 \times S^1}} \colon SWF_R(S^1 \times S^2) \to S^0
    \]
    is the real Bauer--Furuta invariant of $\widetilde{D^3 \times S^1}$, defined with respect to the involution
    \[
    D^3 \times S^1 \to D^3 \times S^1; \qquad (x, y) \mapsto (x, -y),
    \]
    together with the unique real spin structure on this space.
\end{itemize}

It is straightforward to check that:
\[
SWF_R(S^1 \times S^2) = S^0, \qquad
BF^R_{\Sigma_2(S^2 \times D^2, S^2 \times \{0\})} = \pm \id, \qquad \text{ and }\qquad
BF^R_{\widetilde{D^3 \times S^1}} = \pm \id.
\]
By the definition of $|\deg(K)|$, we have
\[
|\deg(BF^R_{\Sigma_2(K)})| = |\deg(K)|.
\]
Combining the identities above, we find
\[
|\deg(BF^R_{\Sigma_2(K)})| = |\deg(BF^R_{\widetilde{S^4 \setminus \nu(K)}})| = |\deg(BF^R_{X(K)})| = |\deg(X(K))|.
\] This completes the proof.
\end{proof}

\begin{cor}
Let $K$ be the $1$-roll spun knot of $P(k(-2,3,7))$, where $P$ is any odd pattern satisfying $P(U) = U$, and $k(-2,3,7)$ denotes the pretzel knot of type $(-2,3,7)$. Then the surgery $X(K)$ does not admit a positive scalar curvature metric.
\end{cor}

\begin{proof}
From the formula in \cite[Section 4 and Theorem 4.32]{Mi23}, we have
\[
|\deg(K)| = |\deg(P(k(-2,3,7)))| = 3.
\]
Applying \cref{degeq}, we obtain
\[
|\deg(X(K))| = |\deg(P(k(-2,3,7)))| = |\deg(k(-2,3,7))| = 3.
\]
The conclusion then follows from \cref{PSC}.
\end{proof}

\noindent Using \cite[Corollary 1.6]{KPT:2024}, one can also treat more general Montesinos knots in a similar way.

\begin{rem} We remark that the known obstructions to the existence of positive scalar curvature metrics on closed 4-manifolds are as follows:
\begin{itemize}
    \item the signature of spin 4-manifolds,
    \item the Seiberg--Witten invariant for 4-manifolds with $b^+ > 1$,
    \item Schoen--Yau's hypersurface method \cite{schoen1987structure},
    \item enlargeability \cite{gromov1980classification}, and the Dirac obstruction \cite{rosenberg1986c}.
\end{itemize}

For example, a homology $S^1 \times S^3$ obtained as a mapping torus of an enlargeable 3-manifold is itself enlargeable, and therefore cannot admit a positive scalar curvature metric. Furthermore, for a homology $S^1 \times S^3$ with a rational homology 3-sphere of a certain class as a cross-section, there exist obstructions \cite{lin2018splitting, lin2019seiberg, konno2020positive, konno2023positive} arising from Seiberg--Witten gauge theory for periodic-end 4-manifolds or from monopole Floer theory applied to the 4-manifold cut open along the cross-section.

Our examples involve surgeries on twisted roll-spun 2-knots. In general, an $m$-twisted $l$-roll spun 2-knot is fibered when $m \neq 0$, in which case the resulting surgery yields a mapping torus. Depending on whether the fibered 3-manifold is enlargeable, this may obstruct the existence of a positive scalar curvature metric. This is why we specifically focus on the case of $0$-twisted $1$-roll spun knots.
\end{rem}

\begin{rem}
It is interesting to compare our invariant $|\deg(X)|$ for a homology $S^1 \times S^3$ manifold $X$ with other numerical invariants of such manifolds arising from Yang--Mills or Seiberg--Witten theory.

There are two well-known invariants of homology $S^1 \times S^3$ manifolds:
\begin{itemize}
    \item the Furuta--Ohta Casson-type invariant $\lambda_{FO}$~\cite{furuta1993differentiable}, defined via instanton counting, and
    \item the Mrowka--Ruberman--Saveliev Casson-type invariant $\lambda_{MRS}$~\cite{mrowka2011seiberg}, defined via Seiberg--Witten solutions with a correction term.
\end{itemize}

These two invariants are conjectured to be equal up to sign. However, when compared with $|\deg(X)|$, one can easily see that $|\deg(X)|$ is independent of them. For example, consider the case where $X = S^1 \times Y$ for an oriented integral homology $3$-sphere $Y$. Then it is easy to check that $|\deg(X)| = 1$, while
\[
\lambda_{FO}(Y) = -\lambda_{MRS}(Y) = \lambda(Y),
\]
where $\lambda(Y)$ denotes the classical Casson invariant of $Y$.
\end{rem}

\bibliographystyle{alpha}
\bibliography{tex}

\end{document}